\documentclass[12pt]{article}
\usepackage[english]{babel}
\usepackage[utf8x]{inputenc}
\usepackage{amsmath, amsthm, amssymb, latexsym, url}
\usepackage{bbm}
\usepackage{xspace}
\usepackage[textwidth=0.89in,textsize=scriptsize]{todonotes}
\newcommand{\NN}{\mathbb{N}}
\newcommand{\ZZ}{\mathbb{Z}}
\newcommand{\RR}{\mathbb{R}}
\newcommand{\QQ}{\mathbb{Q}}
\theoremstyle{definition} \numberwithin{equation}{section}

\newtheorem{theorem}{Theorem}
\newtheorem{lemma}[theorem]{Lemma}
\newtheorem{corollary}[theorem]{Corollary}

\newtheorem{definition}[theorem]{Definition}
\newtheorem{observation}[theorem]{Observation}

\newtheorem{example}[theorem]{Example}
\numberwithin{theorem}{section}

\usepackage{blindtext}
\begin{document}
\title{CP-chains and dimension preservation for projections of (×m,×n)-invariant Gibbs measures}
\author{Javier Ignacio Almarza\thanks{The author was supported by a CONICET doctoral fellowship.}
\thanks{I would like to thank Pablo Shmerkin for introducing this problem to me and for his constant help and advice. I would also like to thank Pablo Ferrari, Mike Hochman and Benjamin Weiss for their helpful remarks.}}

\date{\vspace{-5ex}}
%\date{}  % Toggle commenting to test

\maketitle

\begin{abstract}
Dimension conservation for almost every projection has been well-established by the work of Marstrand, Mattila and Hunt and Kaloshin. More recently, Hochman and Shmerkin used CP-chains, a tool first introduced by Furstenberg, to prove all projections preserve dimension of  measures on $[0,1]^2$ that are the product of a $\times m$-invariant and a $\times n$-invariant measure (for $m$, $n$ multiplicatively independent). Using these tools, Ferguson, Fraser and Sahlsten extended that conservation result to $(\times m,\times n)$-invariant measures that are the pushforward of a Bernoulli scheme under the $(m,n)$-adic symbolic encoding. Their proof relied on a parametrization of conditional measures which could not be extended beyond the Bernoulli case. In this work, we extend their result from Bernoulli measures to Gibbs measures on any transitive SFT. Rather than attempt a similar parametrization, the proof is achieved by reducing the problem to that of the pointwise convergence of a double ergodic average which is known to hold when the system is exact.
\end{abstract}

\section{Introduction}The use of ergodic theory in fractal geometry has recently yielded remarkable advances through the development of a rich theory of processes  of \textit{magnifying} fractal measures. These dynamics, originally introduced by Furstenberg \cite{Furst70}, were recently used by Hochman and Shmerkin \cite{Hochman2010} to prove a conjecture by Furstenberg that states that, for products of sets that are invariant under ``arithmetically independent'' dynamics, all non-trivial projections satisfy dimension conservation. 

\begin{theorem}[Theorem 1.3 of \cite{Hochman2010}]Suppose $\log m/\log n$ is irrational and let $\mu$ and $\nu$ be measures on $\mathbb{T}$ invariant under $T_m$ and $T_n$, respectively. Then
$$\text{dim }\pi(\mu\times\nu)=\min\{1,\text{dim}(\mu\times\nu)\}$$
for all $\Pi_{2,1}\in\setminus\{\pi_1,\pi_2\}$
\end{theorem}

The projections $\pi_1$ and $\pi_2$ are exceptions since they map $\mu\times\nu$ to $\mu$ or $\nu$, and a drop in dimension is to be expected. The result for invariant sets follows from this theorem as a consequence of the variational principle.

Notice the product measures in this theorem are invariant under the non-conformal endomorphism $T_{m,n}:\mathbb{T}^2\to\mathbb{T}^2$ that maps $(x,y)\mapsto (T_m(x),T_n(y))$, and it is natural to raise the question of what kinds of $T_{m,n}$-invariant measures on $\mathbb{T}^2$ satisfy the same dimension conservation result for all projections in $\Pi_{2,1}\setminus\{\pi_1,\pi_2\}$.

Recently, Ferguson, Fraser and Sahlsten \cite{Ferguson2013} proved this was the case for Bernoulli measures, that is, measures on $\mathbb{T}^2$ that are the push-forward of a Bernoulli scheme on the $(m,n)$-adic symbolic encoding of $\mathbb{T}^2$, and raised the question of whether their result could be extended to Gibbs measures. In this context, the term ``Gibbs measures'' is used in the sense introduced by Bowen \cite{Bowen75} to study Anosov diffeomorphisms through symbolic partitions, that is, it refers to measures $m$ on some symbolic space $Y\subseteq \Sigma^\NN$ ($\Sigma$ a finite alphabet) for which there is some constant $K>0$ such that

$$K^{-1}\leq\frac{m(x_1\dots x_k)}{e^{-kP'}e^{\sum_{i=1}^k\phi(T^k(\omega))}}\leq K$$
for some function $\phi$ on $Y$ (called {\em potential}), a constant $P$ and all sequences $\omega=x_1\dots x_ky_{k+1}\dots\in Y$.

In this paper we provide an affirmative answer to Ferguson, Fraser and Sahlsten's question.

\begin{theorem}\label{projectionresult}Let $m,n\in\NN$ satisfy $\log m/\log n\notin\QQ$, $\mu$ be a $T_{m,n}$-invariant measure on $\mathbb{T}$ that is the push-forward of a Gibbs measure for some H\"{o}lder potential and $\pi$ be a projection in $\Pi_{2,1}\setminus\{\pi_1,\pi_2\}$. Then

$$\text{dim }\pi\mu=\min\{1,\text{dim }\mu\}$$
\end{theorem}

The proof, as in their case, is the consequence of showing that such a measure generates a process of magnifying measures of the kind introduced by Furstenberg and used by Hochman and Shmerkin, which are called \textit{CP-chains}. That is, we prove

\begin{theorem}\label{teoprincipal}Let $m,n\in\NN$ satisfy $\log m/\log n\notin\QQ$ and $\mu$ be a $T_{m,n}$-invariant measure on $\mathbb{T}$ that is the push-forward of a Gibbs measure for some H\"{o}lder potential. Then $\mu$ generates an ergodic CP-chain.
\end{theorem}

The projection result follows from this theorem through an application of a major result by Hochman and Shmerkin that relates the dimension of projections of a measure to the dimension of projections of generic measures under the distribution of the CP-process.

Unlike the proof of Theorem \ref{teoprincipal} for Bernoulli schemes in \cite{Ferguson2013},  our proof can't rely on the independence on the past to nicely separate the magnifying CP dynamics on each of the two coordinates of the double torus. Separation, which in \cite{Ferguson2013} is achieved through an elegant parametrization of measures that condition on past coordinates, is important because the two magnifying dynamics must evolve at different speeds to ensure that the sequence of rectangles that the process magnifies into is similar to a sequence of balls, i.e. has bounded eccentricity.

Nevertheless, the Gibbs property ensures enough memorylessness on the past to allow us to reduce the question to a problem of pointwise convergence of double ergodic averages, which is known to have an affirmative answer for exact and Kolmogorov systems, a more general class than Gibbs systems.

It should be noted, though, that our methods do not provide a nice closed expression for the CP-chain distribution that is generated by $\mu$, as was the case in \cite{Ferguson2013}.

The outline of this paper is as follows. Section \ref{cpsection} introduces the machinery of magnifying CP dynamics. Section \ref{sec1} provides the translation of these dynamics on the torus to dynamics on a symbolic space through the $(m,n)$-adic encoding. Section \ref{secergodictheorems} adapts some standard ergodic theorems to our needs and Section \ref{secgibbs} introduces the main properties of Gibbs measures. In Section \ref{secproof} the main result is proved.

\section{Dynamics of conditional probabilities and dimension}\label{cpsection}
\subsection{Preliminaries}
Let us fix the notation and briefly present the main results concerning \textit{CP-chains}, which were first introduced by Furstenberg in \cite{Furst70} and which were recently used by Hochman and Shmerkin in \cite{Hochman2010} to derive strong geometric properties about dimensions of projected measures.

We will be working with probability measures $\mu\in\mathcal{P}(K)$ on some compact metric space $K$, and as usual $\mathcal{P}(K)$ will be endowed with the weak topology, which makes it a compact metrizable space.

For $B\subseteq\RR^d$ any \textit{box}, that is, any product of $d$ real intervals which can be open, closed or half-closed, let $T_B:\RR^d\to\RR^d$ be the homothetic affine transformation that normalizes the volume of $B$ to 1 and sends its ``lower left corner'' to the origin, that is,
$$T_B(x)=\frac{1}{|\text{vol }B|^{1/d}}(x-\min B)$$
where $\min$ refers to the lexicographic ordering.

\begin{definition}For any box $B\subseteq\RR^d$ the \textit{normalized box} is $B^*=T_B(B)$.
\end{definition}

For any Borel probability measure $\mu$ on $\RR^d$ we will write
$$\mu^B=\frac{1}{\mu(B)}\mu\mid_B\circ T_B^{-1}$$
where $\mu\mid_B$ denotes the restriction of $\mu$ to $B$. Clearly, $\mu^B$ is supported on $B^*$.

\begin{definition}Let $\mathcal{E}$ be a collection of boxes in $\RR^d$. A \textit{partition operator} $\Delta$ on $\mathcal{E}$ assigns to each $B\in\mathcal{E}$ a partition $\Delta B\subseteq \mathcal{E}$ of $B$ in a translation and scale-invariant manner, i.e. for any homothety $T$ and $B\in\mathcal{E}$ such that $T(B)\in \mathcal{E}$ then $\Delta(T(B))=T(\Delta(B))$. 
\end{definition}

From a partition operator $\Delta$ we may iteratively define a sequence of refining partitions of $B$.
$$\Delta^0(B)=\{B\}\text{ and }\Delta^{n+1}(B)=\bigcap_{E\in\Delta^n(B)}\Delta(E)$$
\begin{definition}A partition operator $\Delta$ is $\rho$-regular if for any $B\in\mathcal{E}$ there exists a constant $c>1$ such that for any $k\in\NN$ any element $E\in\Delta^k(B)$ contains a ball of radius $\rho^k/c$ and is contained in a ball of radius $c\rho^k$.
\end{definition}

\begin{example}A very natural example of a partition operator is the $m$-adic partition operator $\Delta_m$ on $\mathbb{T}$. Identify $\mathbb{T}$ with $[0,1)$ and let $\mathcal{E}$ be the set of half-open $m$-adic intervals $$\mathcal{E}=\{[l/m^k,(l+1)/m^k):\ k\in\NN;\ l+1\leq m^k\}$$
Then $\Delta_m([l/m^k,(l+1)/m^k))=\{[lm+i/m^{k+1},(lm+i+1)/m^{k+1}):\ i<m\}$. Clearly, this partition is $1/m$-regular.
\end{example}

\begin{definition}A \textit{CP-chain} for a $\rho$-regular partition operator $\Delta$ on a collection of boxes $\mathcal{E}$ is a Markov process $(\mu_k,B_k)_{k=1}^\infty$ with state space
$$\Theta=\{(B,\mu)\in\mathcal{E}\times\mathcal{P}(\RR^d):\ \text{supp }\mu\subseteq B^*\}$$ and transition law given by
$$\text{for $E\in\Delta(B^*)$, $(B,\mu)\mapsto(E,\mu^E)$ with probability $\mu(E)$}$$
\end{definition}

Hence, a CP-chain is a Markov chain of magnified conditional probabilities (whence the ``CP''), that is, measures of the form $\mu^B$ that ``magnify'' a measure $\mu$ into a specific box $B$, chosen according to $\mu$.

In many contexts we will speak of a CP-chain without specifying $\Delta$ or $\mathcal{E}$. Moreover, we will usually have an initial stationary distribution $Q$ on $\Theta$ associated to the chain, and we will use the term ``CP-chain'' to refer not only to the chain (i.e. the transition probabilities) but also to $Q$ and to the shift invariant measure on $\Theta^\NN$ that is induced by $Q$ and the transition law . For such a CP-chain $Q$ the \textit{measure component} will be the projection of $Q$ to its $\mathcal{P}(\RR^d)$ component.

Notice that while the state space sets $\mathcal{P}(\RR^d)$ as the ambient space for measures $\rho$-regularity implies that, conditioned on some initial $B_0\in\mathcal{E}$, the process $(B_n)_{n=1}^\infty$ consists of boxes all of which are included on some compact $K\subseteq\RR^d$, and the measure process $(\mu_n)_{n=1}^\infty$ will take values on a compact and metrizable space $\mathcal{P}(K)$.

To avoid confusion, measures on measure spaces, that is, elements of $\mathcal{P}(\mathcal{P}(K))$, will be called \textit{distributions}.

For any measure $\mu\in\mathcal{P}(K)$ let $\delta_\mu\in\mathcal{P}(\mathcal{P}(K))$ denote the distribution on the space of measures that is supported on the single element $\mu$.

For some fixed box $B\in\mathcal{E}$ and $x\in B$ let us also write $\Delta^k(x)$ for the unique $E\in\Delta^k(B)$ to which $x$ belongs.

\begin{definition}Let $\hat{Q}$ be the measure component of a a CP-chain $Q$ with partition operator $\Delta$ on a collection $\mathcal{E}$ and fix $B\in\mathcal{E}$. We say that $Q$ is \textit{generated} by $\mu\in\mathcal{P}(B^*)$ if for $\mu$-a.e. $x\in B^*$ the \textit{CP scenery distributions}
$$\frac{1}{N}\sum_{k=0}^{N-1}\delta_{\mu^{\Delta^k(x)}}$$
converge weakly to $\hat{Q}$ as $N\to\infty$ and if for any $q\in\NN$ the \textit{q-sparse} CP scenery distributions
$$\frac{1}{N}\sum_{k=0}^{N-1}\delta_{\mu^{\Delta^{qk}(x)}}$$
converge weakly to some possibly different distribution $\hat{Q}_q\in\mathcal{P}(\mathcal{P}(K))$
\end{definition}

It is a consequence of the ergodic theorem that $\hat{Q}$-almost every measure $\nu$ generates $Q$ \cite[see Proposition 7.7]{Hochman2010}.

\subsection{Dimension of measures and CP-chains}
In this section we present a major result of \cite{Hochman2010} connecting the dimension of the projection of a measure to the average ``local entropy dimension'' of such a projection under a CP-chain distribution generated by the original measure. We denote the Hausdorff dimension of a set by $\text{dim }A$.

\begin{definition}Let $\mu\in\mathcal{P}(X)$ be a measure on some metric space $X$. The \textit{lower Hausdorff dimension} of $\mu$ is defined as
$$\text{dim}_*\mu=\inf\{\text{dim }A:\ \mu(A)>0\}$$

The \textit{upper and lower local dimensions} of $\mu$ at a point $x$ are defined as
$$\overline{\text{dim}}(\mu,x)=\limsup_{r\to 0}\frac{\log \mu(B_r(x))}{\log r}$$
and
$$\underline{\text{dim}}(\mu,x)=\liminf_{r\to 0}\frac{\log \mu(B_r(x))}{\log r}$$
respectively.

We say that  $\mu$ is exact dimensional if for $\mu$-a.e. $x$ we have that $\overline{\text{dim}}(\mu,x)=\underline{\text{dim}}(\mu,x)=c$ for some constant $c$. We observe that this implies $\text{dim}_*\mu=c$ (see \cite[proposition 10.2]{FalconerTech}). We may then unambiguously write $\dim \mu$ when $\mu$ is exact dimensional. 

Finally, for a distribution $\hat{Q}\in\mathcal{P}(\mathcal{P}(X))$ let the \textit{lower dimension} of $\hat{Q}$ be
$$\dim_*\hat{Q}=\int\dim_*\nu\ d\hat{Q}(\nu)$$
\end{definition}

It can be seen that the ergodic theorem implies that $\hat{Q}$-almost every measure $\nu$ is exact dimensional when $\hat{Q}$ is the measure component of a CP-chain (see \cite[Lemma 7.9]{Hochman2010}), so that in this case we may unambiguously speak of the dimension $\dim Q=\dim \hat{Q}$ of the CP-chain.

Let $\Pi_{d,k}$ denote the set of orthogonal projections from $\RR^d$ to any $k$-dimensional subspace.

\begin{definition}For any measure $\mu$ the \textit{r-entropy} of $\mu$ is defined as
$$H_r(\mu)=-\int\log\mu(B(x,r))d\mu(x)$$

Let $Q$ be a CP-chain for a $\rho$-regular partition operator $\Delta$ and $q>0$. Then $E_q:\Pi_{d,k}\to\RR$ is the function defined by
$$E_q(\pi)=\int\frac{1}{q\log(1/\rho)}H_{\rho^q}(\pi\nu)d\hat{Q}(\nu)$$
\end{definition}

\begin{theorem}\cite[Theorem 8.2]{Hochman2010}\label{teohochman}Let $Q$ be an ergodic CP-chain. Then the function $E(\pi)=\lim_{q\to 0}E_q(\pi)$ exists pointwise, is lower semi-continuous
\footnote{For the quotient topology on $\Pi_{d,k}=\text{GL}(\RR^d)/H$, regarded as a homogeneous space for the Lie group $\text{GL}(\RR^d)$, and $H$ the stabilizer of any $k$-dimensional subspace of $\RR^d$.}
%\footnote{For the topology that $\Pi_{d,k}$ inherits from the Lie group $\text{GL}(\RR^d)$ when identified with the homogeneous space $\text{GL}(\RR^d)/H$, where $H$ is the stabilizer of any $k$-dimensional subspace of $\RR^d$ for the natural action of $\text{GL}(\RR^d)$ on the set of $k$-dimensional subspaces.} and satisfies:
\begin{enumerate}
\item For any $\pi\in\Pi_{d,k}$ and $\hat{Q}$-almost every $\nu$
$$\text{dim }\pi\nu=E(\pi)$$
\item For almost every $\pi\in\Pi_{d,k}$ we have
$$E(\pi)=\min(k,\text{dim }Q)$$
\item For any measure $\mu$ that generates $Q$ and for any $\pi\in\Pi_{d,k}$
$$\text{dim }\pi\mu\geq\text{dim }E(\pi)$$
\end{enumerate}
\end{theorem}

\section{The symbolic dynamical setting}\label{sec1}
\subsection{The $(m,n)$-encoding}Since we will be working with measures on $\mathbb{T}^2$ that are $T_{m,n}$-invariant we will at first try to build CP-chains for the $(m,n)$-adic partition operator $\Delta=\Delta_{m,n}$ that is the product of $\Delta_m$ and $\Delta_n$ (that is, the first $\mathbb{T}$-coordinate is partitioned acccording to $\Delta_m$, the second according to $\Delta_n$). As in \cite{Ferguson2013}, we will find it convenient to translate these CP-chains for measures supported on real boxes to CP-chains for measures supported on symbolic spaces, that is, on spaces of sequences $\Sigma^\NN$ for some finite alphabet $\Sigma$. We won't actually define a symbolic CP-chain (that is carried out in \cite{Furst08}), but the definition arises naturally and is implicit in our construction. Our alphabet will be $\Sigma=\Lambda\times\Lambda'$ where $\Lambda$ and $\Lambda'$ are finite sets of size $m$ and $n$, respectively, which may be identified with $\{0,\dots,m-1\}$ and $\{0,\dots,n-1\}$. By $\Lambda^*$ and $\Lambda'^*$ we will denote the free monoids generated by $\Lambda$ and $\Lambda'$, and we will consider the spaces of infinite sequences $\Omega_+=\Sigma^\NN$ and $\Omega=\Sigma^\ZZ$ endowed with their product topologies and their Borel $\sigma$-algebras. We will write $\pi:\Omega\to\Omega_+$ for the projection that erases the ``past'' coordinates. Small latin letters $a$, $b$ and $c$ will be used for elements in $\Lambda^*$ or $\Lambda'^*$ and $\omega=(\omega^1,\omega^2)$ will denote any element in $\Omega$ or $\Omega_+$. The length of any word $a\in\Lambda^*$ will be written $|a|$, and the restriction of $\omega\in\Omega_+$ (or $\omega^1\in\Lambda$) to its first $k$ values will be denoted $\omega_{|k}$.

We have an onto map $\xi_m:\Lambda^\NN\to[0,1]$ given by $\xi_m(\omega^1)=\sum_{k=1}^\infty\omega_k^1m^{-k}$ and similarly for $\xi_n$, so that we may define $\xi:\Omega_+\to[0,1]^2$ given by $\xi(\omega)=(\xi_m(\omega^1),\xi_n(\omega^2))$.

To each word $a=\omega^1_1\dots \omega ^k_1\in\Lambda^k$ and $b=\omega^2_1\dots \omega^2_k\in\Lambda'^{k}$ we may assign a unique box
$$R(a,b)=[\xi_m(a 0^\infty),\xi_m(a 0^\infty)+m^{-k})\times[\xi_n(b0^\infty),\xi_n(b0^\infty)+n^{-k})\in\Delta^k(\mathbb{T} ^2)$$
where $a0^\infty$ is the infinite sequence which consists of the prefix $a$ followed by an infinite sequence of zeros. This assignment is a bijection from $\Sigma^k$ to $\Delta^k(\mathbb{T}^2)$. Clearly,
$$R(\omega_{|k}^1,\omega_{|k}^2)=\Delta^k(\xi(\omega))$$
and
$$\{\xi(\omega)\}=\bigcap_kR(\omega_{|k}^1,\omega_{|k}^2)$$

\subsection{Bounding eccentricities}

The problem with the sequence of refining partitions $\Delta^k(\mathbb{T}^2))$ is that it is not $\rho$-regular for any $\rho>0$, since the eccentricity of any box $B\in\Delta^k(\mathbb{T}^2)$ (defined as the ratio between its longest and shortest side) is $(n/m)^k$. To bound this eccentricity we will have to slightly change our partition operator to $\Delta'$, where, for an $m$-adic interval $J_1$, a $n$-adic interval $J_2$ and $B=J_1\times J_2$, $\Delta'(B)$ always partitions $J_1$ into $m$ subintervals and then partitions $J_2$ into $n$ subintervals if and only if not doing so would result in a box of eccentricity greater than $n$. Writing $\text{ecc}(B_k)$ for the eccentricity of any box $B_k\in\Delta'^k(\mathbb{T}^2)$, this rule means that, in the first case, the boxes $B_{k+1}$ will have eccentricity $\text{ecc}(B_{k+1})=m.\text{ecc}(B_k)$, and in the second case the boxes $B_{k+1}$ will have eccentricity $\text{ecc}(B_{k+1})=m.\text{ecc}(B_k)/n$. 

This process can be encoded by translating eccentricities into angles via $t=\frac{\log (\text{ecc}B)}{\log (n)}$. Let $\alpha=\log m/\log n$. Our partition operator then sends the angle $t$ to $t+\alpha$ when $t+\alpha<1$, and sends $t$ to $t+\alpha-\log n$ when that condition is not met. Equivalently, our partition operator induces a rotation by angle $\alpha$ on $\mathbb{T}$.

It follows that any element of $\Delta'^k(\mathbb{T}^2)$ will be congruent to
$$R_t=[0,1)\times[0,e^{t\log n})$$
which has eccentricity bounded in $[1,n]$. Thus, our modified partition operator $\Delta'$ is $(1/m)$-regular.

\subsection{Measures and CP dynamics on symbolic spaces}\label{subsecmeasurescp}

Back to the symbolic model, elements of $\Delta'^k(\mathbb{T}^2)$ no longer correspond to words in $\Sigma^k$. They correspond to words in $\Lambda^k\times\Lambda'^{l_k(0)}$, where $l_k(t)=\lfloor t+\alpha k\rfloor$ (we will also write $l_k=l_k(0)$) (see \cite[Section 3.3]{Ferguson2013} for more details). Thus
$$R(\omega_{|k}^1,\omega_{|l_k}^2)=\Delta'^k(\xi(\omega))$$

Let us write $T$ for the shift transformation on any symbolic one-sided or two-sided sequence space.

The measures $\mu'$ on $\mathbb{T}^2$ (or $\RR^2$) we will consider are push-forwards of measures $\mu$ on $\Omega_+$, that is, $\mu'=\mu\circ\xi^{-1}$, and if $\mu$ is $T$-invariant then $\mu'$ is $T_{m,n}$-invariant, since $T_{m,n}\xi=\xi T$.

Let $\mu$ be a $T$-invariant, weak mixing measure on $\Omega_+$. By some abuse of notation $\mu$ will also denote its natural extension to $\Omega$.

Let $\mathcal{F}_i^j$ be the $\sigma$-algebra on $\Omega$ generated by vectors $\omega\mapsto\omega^1_i,\dots,\omega^1_j$ ($i,j\in\ZZ\cup\{-\infty,\infty\}$). Similarly, $\mathcal{G}_i^j$ will be the $\sigma$-algebra generated by $\{\omega^2_k\}_{i\leq k\leq j}$.

Words $a$ in $\Lambda^*$ may be identified with cylinder sets $[a]$. These word cylinders generate the algebra of clopen cylinder sets, denoted $G(\Lambda^*)$. Write $\pi_\Lambda:\Lambda^\ZZ\to\Lambda^\NN$ for the projection that erases past coordinates (and similarly for $\pi_{\Lambda'}$). If $C$ is any measurable subset of $\Lambda^\NN$ (or $\Lambda'^\NN$) we will write $\overline{C}=\pi_\Lambda^{-1}(C)\times\Lambda'^\ZZ$ for the corresponding subset of $\Omega$ (respectively,  $\overline{C}=\Lambda^\mathbb{Z}\times\pi_{\Lambda'}^{-1}(C)$). To any clopen $A\in G(\Lambda^*)\times G(\Lambda'^*)$ associate the following continuous functional on $\mathcal{P}(\Omega_+)$ (continuity follows from the fact that $\mathbbm{1}_A$ is a continuous function on $\Omega_+$ when $A$ is a cylinder): 
$$\phi_{A}(\nu)=\nu(A)$$

We will write $\mathcal{I}$ for the collection of finite unions of intervals in $\mathbb{T}$.

Let $\Xi=\{(\omega,\nu)\in\Omega_+\times \mathcal{P}(\Omega_+):\ \omega\in\text{supp}\ \mu\}$ and $X=\mathbb{T}\times\Xi$.

Given any measure $\nu\in \mathcal{P}(\Omega_+)$, we may consider conditional measures $\nu_{t,k}^\omega$, for $\omega$ in $\text{supp}\ \nu$, defined as
$$\nu_{t,k}^{\omega}(A\times B)=\nu\left(T^{-k}(\overline{A})\cap T^{-l_k(t)}(\overline{B})\mid \ \mathcal{F}_1^k\vee\mathcal{G}_1^{l_k(t)}\right)(\omega)$$ 
on product cylinders $A\times B\in G(\Lambda^*)\times G(\Lambda'^*)$ and extended to the Borel $\sigma$-algebra using Caratheodory's theorem.

Since we will often shift conditional measures and expectations, it will be convenient to recall that, for $\nu$ any $T$-invariant measure on $(X,\mathcal{B})$, the following is true $\nu$-a.s. for any $\sigma$-algebra $\mathcal{F}\subseteq\mathcal{B}$ and any $\mathcal{B}$-measurable $f$

\begin{equation}\label{shiftedcond}\mathbb{E}_\nu[f\circ T^k\mid\mathcal{F}]=\mathbb{E}_\nu[f\mid T^{-k}(\mathcal{F})]\circ T^k
\end{equation}

Write $\{x\}=x-\lfloor x\rfloor$ for any real $x$ and $R_\beta:[0,1)\to[0,1)$ ($\beta$ any real) for the rotation transformation $R_\beta(x)=\{x+\beta\}$.

Let $S:X\to X$ be the \textit{symbolic magnification operator}, defined by
$$S(t,\omega,\nu)=(R_\alpha(t),T_{t,1}(\omega),S_{t,1}(\omega,\nu))$$
where $T_{t,k}(\omega)=(T^k(\omega^1),T^{l_k(t)}(\omega^2)$ and $S_{t,k}:\Xi\to \mathcal{P}(\Omega_+)$ is given by
$S_{t,k}(\omega,\nu)=\nu_{t,k}^\omega$

The following two observations follow from (\ref{shiftedcond}) and the definition of $S$

\begin{observation}\label{obseboba}$$S^k(t,\omega,\nu)=(R_\alpha(t),T_{t,k}(\omega),S_{t,k}(\omega,\nu))$$
\end{observation}

\begin{observation}\label{obsemedidas}
$$S_{s,h}\left(T_{t,k}(\omega),\nu_{t,k}^\omega\right)(A\times B)=\mathbb{E}_\nu\left[T^{k+h}(\mathbbm{1}_{\overline{A}})T^{l_{k,h}(t,s)}(\mathbbm{1}_{\overline{B}})\mid \mathcal{F}_1^{k+h}\vee\mathcal{G}_1^{l_{k,h}(t,s)}\right](\omega)$$
where $l_{k,h}(t,s)=l_k(t)+l_h(s)$.
\end{observation}

To show that the push-forward of some $\mu$ in an appropriate class of measures generates a CP-chain it will suffice to show that for any $f\in C(X)$ the averages
\begin{equation}\label{genericity}\frac{1}{N}\sum_{k=1}^Nf(S^{qk}(t,\omega,\mu))
\end{equation}
converge $\mu$-a.e. when $t=0$ to some constant $c_f$. The linearity of $c_f$ with respect to $f$ and the obvious bound $c_f\leq\|f\|_\infty$ define a continuous functional on $C(X)$ and the Riesz representation theorem then implies the existence of a measure $Q$ in $X$ such that $c_f=\mathbb{E}_Q[f]$. From (\ref{genericity}) it is evident $Q$ is $S$-invariant and we will also show it is $S$-ergodic. Indeed, the following theorem summarizes the results proven in sections \ref{secgenericity} and \ref{secergodicity}.

\begin{theorem}\label{genericity2}Let $\mu$ be a Gibbs invariant measure for some topologically transitive subshift of finite type $Y\subseteq\Omega_+$. Then
there is a distribution $\hat{Q}$ on $\Xi$ such that $Q=\lambda\times\hat{Q}$ is an $S$-invariant and ergodic distribution on $X$ and for every continuous $f\in C(X)$, every $q\in\NN$ and $\mu$-a.e. $\omega\in\Omega_+$ we have
\begin{equation}\frac{1}{N}\sum_{k=1}^Nf(S^{qk}(0,\omega,\mu))\to\mathbb{E}_Q[f]
\end{equation}
\end{theorem}

\subsection{From the symbolic to the geometric model} We will now show how Theorem \ref{genericity2} implies Theorem \ref{teoprincipal}. Recall first the following definition.

\begin{definition}Let $(X,\mu,T)$ be a measure-preserving system. We say that the system is {\em exact} if, for any finite partition $\mathcal{A}$,
\begin{equation}\label{defexact}Tail(\mathcal{A}):=\bigwedge_{n=0}^{\infty}\bigvee_{k=n}^\infty T^{-k}(\mathcal{A})=\{\emptyset,X\}\quad\text{(mod $\mu$)}
\end{equation}

If $T$ is an automorphism and (\ref{defexact}) holds, then we say that the $\mathbb{Z}$-system $(X,\mu,T)$ is a {\em Kolmogorov system} (or just {\em K system}).
\end{definition}

We have the following fact (see \cite[Theorem 4.2.12]{Urbanski09})

\begin{lemma}\label{lemagibbsexact} Let $\mu$ be a Gibbs invariant measure for a H\"{o}lder potential $\phi$ on some topologically transitive subshift of finite type $Y$. Then $(Y,\mu,T)$ is exact and its natural extension to a $\mathbb{Z}$-system is Kolmogorov.
\end{lemma}

We are now ready to show the implication. Let $E$ be the set of boundaries of all boxes obtained through the partition sequence $\Delta^k(\mathbb{T}^2)$, i.e.,
$$E=\bigcup_{a\in\Lambda^*,b\in\Lambda'^*}\partial R(a,b)$$
Let $\mu'$ be the push-forward of $\mu$, that is $\mu'=\xi_*(\mu)$ (where for any $g$, $g_*(\mu)=\mu\circ g^{-1}$) for some shift-invariant $\mu$ on $\Omega_+$ for which the system $(\Omega_+,\mu,T)$ is exact (by Lemma \ref{lemagibbsexact} this applies in particular to Gibbs measures).

Let us assume first that $\mu$ assigns positive measure to some set of the form $\{\omega\in\Omega_+:\ \omega^1=\tilde{a}0^\infty\}$ for a fixed $\tilde{a}\in\Lambda^*$, which is the same as saying that $\mu'$ assigns positive measure to some vertical line that intersects the horizontal axis at an $m$-adic point. This also implies $\mu'(E)>0$.

Set $A_k=\{\omega\in\Omega_+:\ \omega^1=a0^\infty \text{ for some }a\in\Lambda^k\}$. Notice that the $A_k$ form a monotonically increasing sequence of events. By our assumption, for any $k\geq |\tilde{a}|$, $\mu(A_k)>0$. Then $A=\bigcup_{k}A_k$ is a tail event, so that by exactness we must have $\mu(A)=1$, which means almost every sequence has its $m$-coordinate eventually terminated by an infinite sequence of zeros. Equivalently, $\mu'$ is supported on $m$-adic vertical lines.

Now, take any word $a=a_1\dots a_l\in\Lambda^*$ containing at least one nonzero symbol. By invariance, for any $k\in\NN$,
$$\mu(\{\omega:\ \omega^1_1=a_1\dots \omega_l^1=a_l\})=\mu(\{\omega:\ \omega^1_{1+k}=a_1\dots \omega_{l+k}^1=a_l\})\leq 1-\mu(A_k)$$
and given any $\epsilon>0$ we may pick $k_0$ such that for all $k\geq k_0$, $\mu(A_k)\geq 1-\epsilon$, so that, by choosing any arbitrarily small $\epsilon$ we get $\mu(\{\omega:\ \omega^1_1=a_1\dots \omega_l^1=a_l\})=0$. Hence, $\mu$ is supported on sequences whose $m$-coordinate is $0^\infty$. Equivalently, $\mu'$ is a supported on a single vertical line, and the problem reduces to that of a single $T_n$-invariant measure on the unit interval, i.e., a conformal measure, for which genericity results have already been established under much more general assumptions in \cite{HochmanJEMS}.
The same reasoning can be carried out for horizontal $n$-adic lines, and these two cases exhaust the possible cases in which $\mu'(E)>0$.

Thus, we may now assume that $\mu'(E)=0$.

Since the modified partition operator $\Delta'$ is $(1/m)$-regular we know all boxes in the sequence $\Delta'^k([0,1))$ will lie within some compact box, which we denote $\overline{R}$, and the measures $\mu'^{\Delta'^k(x)}$ for $\mu'$-a.e. $x$ will be supported in $\overline{R}$.

Define $P:X\to\mathcal{P}(\overline{R})$ as $P(t,\omega,\nu)=S_{t*}(\xi_*(\nu))$, where $S_t:\RR^2\to\RR^2$ is the linear transformation that maps $[0,1]$ to the normalized box of eccentricity $e^{t\log n}$, $R_t^*$. Since $\mu'(E)=0$ we have 
$$\mu'^{\Delta'^{qk}(\xi(\omega))}=\mu'^{R\left(\omega^1_{|qk},\omega^2_{|l_{qk}}\right)}=S_{\{qk\alpha\}}\left(\xi_*(\mu^\omega_{0,qk})\right)$$

Notice that if $\mu'(E)>0$ then $\mu'(R(\omega^1_{|k},\omega^2_{|l_k}))$ may be strictly greater than $\mu([\omega^1_{|k}]\times[\omega^2_{|k}])$, since the box may contain a segment of strictly positive measure on its boundary, and these elements, having two possible representations, will add the weight of the cylinders of both representations to the pushforward measure $\mu'$.

Now, $\mu'^{\Delta'^{qk}(\xi(\omega))}=P(S^{qk}(0,\omega,\mu))$ and while $P$ is not continuous in the $\mathbb{T}$-coordinate when $t=0$ this is the only way a discontinuity can occur, so that for any continuous $f\in C(\mathcal{P}(\overline{R}))$ the composition $f\circ P$ will be continuous on $X$ except for a set of measure zero. We use the following result, which is proved in \cite[Theorem 2.7]{Billings}
\begin{lemma}Let $K$ be a compact metric space. If $\nu_N\to\nu$ weakly in $\mathcal{P}(K)$ and $f:K\to\RR$ is $\nu$-almost everywhere continuous then $\int fd\nu_N\to\int fd\nu$.
\end{lemma}

It follows from Theorem \ref{genericity2} that for any Gibbs $\mu$ and $\mu$-a.e. $\omega$
$$\frac{1}{N}\sum_{k=0}^{N-1}(f\circ P)(0,\omega,\mu)\to\int f\circ Pd(\lambda\times P)$$

This completes the proof of Theorem \ref{teoprincipal} from Theorem \ref{genericity2}.

%\subsection{From measures to sets}

\section{Ergodic theorems}\label{secergodictheorems}

We will now adapt to our needs some standard theorems on the convergence of nonconventional ergodic averages. Throughout this section $(\Omega,\mu,T)$ will denote a measure-preserving system on a metric space $\Omega$ with its Borel $\sigma$-algebra and $\lambda$ will denote Lebesgue measure on $\mathbb{T}$. As usual, for any measurable $f$, we will write $T(f)$ for the Koopman's operator $T(f)=f\circ T$.

\begin{lemma}\label{lemaaux}Let $(\Omega,\mu,T)$ be weak mixing , let $\alpha\in[0,1)$ be irrational and let $F:\mathbb{T}\times\Omega\to\RR$ be an integrable function (for the product measure $\lambda\times\mu$) such that the family of functions $(F(\cdot,\omega))_{\omega\in\Omega}$ is uniformly equicontinuous. Then there is a full measure $\Omega'\subseteq\Omega$ such that for all $\omega\in\Omega'$ and all $t\in\mathbb{T}$

$$\frac{1}{N}\sum_{k=1}^NT_\alpha^{k}(F)(t,\omega)\to\mathbb{E}_{\lambda\times\mu}(F)$$
where $T_\alpha=R_{\alpha}\times T$.

\end{lemma}
\begin{proof}
The system $(\mathbb{T}\times\Omega,\lambda\times\mu,T_\alpha)$, is ergodic, since $(\Omega,\mu,T)$ is weak mixing. Then there is a full $\lambda\times\mu$-measure subset $X\subseteq \mathbb{T}\times\Omega$ on which
$$\frac{1}{N}\sum_{k=1}^NT_\alpha^{k}(F)(t,\omega)\to\mathbb{E}_{\lambda\times\mu}(F)$$

In particular, there is a countable and dense subset  $\mathbb{T}'\subseteq\mathbb{T}$ such that for all $t\in\mathbb{T}'$ the section $X_t=\{\omega\in\Omega:\ (t,\omega)\in X\}$ has full $\mu$-measure. We take $\Omega'=\bigcap_{t\in\mathbb{T}'}X_t$, which has full measure.

Now, if we fix $\epsilon>0$, by uniform equicontinuity there is some $\delta>0$ such that if $|t-t'|<\delta$ then $|F(t,\omega)-F(t',\omega)|<\epsilon$ for all $\omega$. Let $t\in\mathbb{T}$ and take some $t'\in\mathbb{T}'$ such that $|t-t'|<\delta$. Notice that this implies $|R_\alpha^k(t)-R_\alpha^k(t')|<\delta$ for all $k$ (and hence $|F(R_\alpha^k(t),T^k(\omega))-F(R_\alpha^k(t'),T^k(\omega))|<\epsilon$) since rotation is an isometry. Then, for any $\omega\in\Omega'$
\begin{align}\frac{1}{N}\left|\sum_{k=1}^NT_\alpha^{k}(F)(t,\omega)-\mathbb{E}_{\lambda\times\mu}(F)\right|&\leq\frac{1}{N}\left|\sum_{k=1}^NT_\alpha^{k}(F)(t,\omega)-T_\alpha^{k}(F)(t',\omega)\right|+\nonumber\\
&+\frac{1}{N}\left|\sum_{k=1}^NT_\alpha^{k}(F)(t',\omega)-\mathbb{E}_{\lambda\times\mu}(F)\right|\leq\nonumber\\
&\leq\epsilon+\frac{1}{N}\left|\sum_{k=1}^NT_\alpha^{k}(F)(t',\omega)-\mathbb{E}_{\lambda\times\mu}(F)\right|
\end{align}
And since $(t',\omega)\in X$ the $N$ can be chosen large enough so that the summand on the right is arbitrarily small.
\end{proof}

\begin{theorem}\label{teo1}Let $(\Omega,\mu,T)$ be weak mixing, let $\alpha\in[0,1)$ be irrational. Let $I\in\mathcal{I}$ and $F$ be some bounded, measurable function.

Then there is a full measure $\Omega'\subseteq\Omega$ such that, for all $t\in[0,1)$,
$$\frac{1}{N}\sum_{k=1}^N\mathbbm{1}_I(R_\alpha(t))T^{l_k(t)}(F)(\omega)\to|I|\ \mathbb{E}\left[F\right]$$
for all $\omega\in\Omega'$.
%and $\{\omega:\ (0,\omega)\in\Omega'\}$ has full $\mu$-measure.
\end{theorem}
\begin{proof}Notice that since $I$ is a finite union of intervals we can write $\mathbbm{1}_I=\sum_i^r\mathbbm{1}_{I_i}$ for some $r\geq 1$ and intervals $I_i$ such that $|I_i|\leq \epsilon<\alpha$. Thus, by linearity, we may assume that $I=(a,a+\epsilon)$ is an interval such that $\epsilon<\alpha$.

Write $k_1(t)<k_2(t)<\dots<k_i(t)<\dots$ for all $k$ such that $t+\alpha k\in I$. Notice that for any $k\in\NN$ there is some $i$ such that $k=k_i(t)$ if and only if there is some (unique) $n\in\NN$ such that
\begin{equation}\label{rotation}
n+a-t\leq k\alpha\leq n+a-t+\epsilon
\end{equation}
We will write $n_i(t)$ for the unique $n\in\NN$ that satisfies (\ref{rotation}) with $k=k_i(t)$. Then $n_i(t)=\lfloor t+\alpha k_i(t)\rfloor$, and since $|I|<\alpha$ we have $l_{k_i}(t)=\lfloor t+\alpha k_i(t)\rfloor\neq\lfloor t+\alpha k_j(t)\rfloor=l_{k_j}(t)$ when $i\neq j$. Write $c(t)=\{(a-t)/\alpha\}$. It follows from these observations that $n=[t+\alpha k_i]$ for some $i$ if and only if $\{c(t)+n/\alpha\}\in (1-\epsilon/\alpha,1)$.

Consider the system $(\mathbb{T}\times\Omega,\lambda\times\mu,T')$ with $T'=R_{\frac{1}{\alpha}}\times T$. This system is ergodic, since $(\Omega,\mu,T)$ was weak mixing.

Write $I'=(1-\epsilon/\alpha,1)$ and $F'(t,\omega)=\mathbbm{1}_{I'}(t)F(\omega)$. Then
\begin{equation}\label{ergeq}
\frac{1}{N}\sum_{k=1}^N\mathbbm{1}_I(R_\alpha(t))T^{l_k(t)}(F)(\omega)=\frac{1}{N}\sum_{n=1}^{\lfloor t+\alpha N-a\rfloor} T'^{n}(F')(c(t),\omega)
\end{equation}

We want to show that
\begin{equation}\label{ergeq2}
\frac{1}{\lfloor t+\alpha N-a\rfloor}\sum_{n=1}^{\lfloor t+\alpha N-a\rfloor}T'^{n}(F') (c(t),\omega)\to\mathbb{E}_{\lambda\times\mu}[F']=\frac{\epsilon}{\alpha}E[F]
\end{equation}
since (\ref{ergeq}) would then imply
$$\frac{1}{N}\sum_{k=1}^N\mathbbm{1}_I(R_\alpha(t)) T^{l_k(t)}(F)(\omega)\to\frac{\epsilon}{\alpha}E[F]\lim_N\frac{\lfloor t+\alpha N-a\rfloor}{N}=\epsilon E[F]$$

Now, $\mathbbm{1}_{I}$ is Riemann integrable, that is, it can be approximated by continuous functions $g_m\leq \mathbbm{1}_{I}\leq h_m$ such that $\int h_m(t)-g_m(t)d\lambda(t)<2^{-m}$, and notice that since $g_m$ is continuous in a compact space and $F$ is bounded, $G_m(t,\omega)= g_m(t)F(\omega)$ satisfies the uniform equicontinuity hypothesis of Lemma \ref{lemaaux}, and the same is true of $H_m= h_m(t)F(\omega)$. Then there are full $\mu$-measure sets $\Omega_m$ such that for all $t\in[0,1)$ and $\omega\in\Omega_m$ (\ref{ergeq2}) obtains with $G_m$ or $H_m$ substituting for $F'$. Fix $M$ such that $|F|<M$, assume first that $F$ is nonnegative and write $N_\alpha=\lfloor t+\alpha N-a\rfloor$. Then

$$\frac{1}{N_\alpha}\sum_{n=1}^{N_\alpha}T'^{n}(G_m) (c(t),\omega)\leq\frac{1}{N_\alpha}\sum_{n=1}^{N_\alpha}T'^{n}(F') (c(t),\omega)\leq\frac{1}{N_\alpha}\sum_{n=1}^{N_\alpha}T'^{n}(H_m) (c(t),\omega)$$

and taking limits for $\omega\in\Omega'=\bigcap_m\Omega_m$, we get
\begin{align}
&\mathbb{E}_{\lambda\times\mu}[G_m]\leq\liminf_N\frac{1}{N_\alpha}\sum_{n=1}^{N_\alpha}T'^{n}(F') (c(t),\omega)\leq\nonumber\\
&\leq\limsup_N\frac{1}{N_\alpha}\sum_{n=1}^{N_\alpha}T'^{n}(F') (c(t),\omega)\leq\mathbb{E}_{\lambda\times\mu}[H_m]
\end{align}

But $\mathbb{E}_{\lambda\times\mu}[H_m]-\mathbb{E}_{\lambda\times\mu}[G_m]\leq M2^{-m}$, so taking $m$ to infinity we get

$$\liminf_N\frac{1}{N_\alpha}\sum_{n=1}^{N_\alpha}T'^{n}(F') (c(t),\omega)=\limsup_N\frac{1}{N_\alpha}\sum_{n=1}^{N_\alpha}T'^{n}(F') (c(t),\omega)$$

and the limit is $\lim_m \mathbb{E}_{\lambda\times\mu}[H_m]=\mathbb{E}_{\lambda\times\mu}[F']$. This proves (\ref{ergeq2}) for nonnegative $F$. When $F$ has a negative part we write $F=F^+-F^-$ for nonnegative $F^+$ and $F^-$, and we have $F'=F'^+-F'^-$, where $F'^+=\mathbbm{1}_{I'}F^+$ and $F'^-=\mathbbm{1}_{I'}F^-$ are nonnegative, so that convergence for $F'^+$ and $F'^-$ entails that of $F'$.

%Then, setting $\Omega'=\bigcap_n\Omega_n$.

%\todo{Incluir demostración}

%Now, for $F\in L^\infty$ and $n>0$ one has a Riemann integrable (and bounded) $F_n$ such that $\mathbb{E}[|F-F_n|]\leq 2^{-n}$. Then, for $\omega\in\Omega_*$
%\begin{align}&\limsup_N\left|\frac{1}{N}\sum_{k=1}^N\mathbbm{1}_I(R_\alpha(t))T^{l_k(t)}(F)(\omega)-|I|\ \mathbb{E}[F]\right|\leq\nonumber\\
%&\leq\limsup_N\left|\frac{1}{N}\sum_{k=1}^N
%T^{l_k(t)}(F-F_n)(\omega)\right|+|I|| \mathbb{E}[F]-\mathbb{E}[F_n]|+\nonumber\\
%&+\limsup_N\left|\frac{1}{N}\sum_{k=1}^N\mathbbm{1}_I(R_\alpha(t))T^{l_k(t)}(F_n)(\omega)-|I|\ \mathbb{E}[F_n]\right|=\nonumber\\
%&=\limsup_N\left|\frac{1}{N}\sum_{k=1}^N
%T^{l_k(t)}(F-F_n)(\omega)\right|+|I|| \mathbb{E}[F]-\mathbb{E}[F_n]|
%\end{align}
%From Theorem 2 in \cite{Bour89} we know the first summand converges almost surely to $|\mathbb{E}[F]-\mathbb{E}[F_n]|$, that is, there is a full $\mu$-measure set $\Omega_n$ such that
%$$\limsup_N\left|\frac{1}{N}\sum_{k=1}^N\mathbbm{1}_I(R_\alpha^k(t))T^{l_k(t)}(F)(\omega)-|I|\ \mathbb{E}[F]\right|\leq 2|I|| \mathbb{E}[F]-\mathbb{E}[F_n]|<2^{-(n-1)}$$
%Then, on $\Omega'=\bigcap_n\Omega_n$
%$$\limsup_N\left|\frac{1}{N}\sum_{k=1}^N\mathbbm{1}_I(R_\alpha^k(t))T^{l_k(t)}(F)(\omega)-|I|\ \mathbb{E}[F]\right|=0$$
%and $\mu(\Omega')=1$.
\end{proof}
We now obtain a corollary analogous to Maker's generalized ergodic theorem.
\begin{corollary}\label{coro2}Let $(\Omega,\mu,T)$ be weak mixing, let $\alpha\in[0,1)$ be irrational. Let $I\in\mathcal{I}$ and $(F_k)_{k\in\NN}$ be some sequence of uniformly bounded measurable functions on $\Omega$ such that $F_k\to F$ almost surely.

Then there is a full measure $\Omega'\subseteq\Omega$ such that, for all $t\in[0,1)$,
$$\frac{1}{N}\sum_{k=1}^N\mathbbm{1}_I(R_\alpha(t))T^{l_k(t)}(F_k) (\omega)\to|I|\ \mathbb{E}\left[F\right]$$
\end{corollary}
\begin{proof}Write $G_M(\omega)=\sup_{k\geq M}|F(\omega)-F_k(\omega)|$. We have
\begin{align}&\limsup_N\left|\frac{1}{N}\sum_{k=1}^N\mathbbm{1}_I(R_\alpha(t))T^{l_k(t)}(F_k) (\omega)-|I|\ \mathbb{E}\left[F\right]\right|\leq\nonumber\\
&\leq\limsup_N\frac{1}{N}\sum_{k=1}^NT^{l_k(t)}(|F_k-F|)(\omega)+\nonumber\\
&\qquad +\limsup_N\left|\frac{1}{N}\sum_{k=1}^N\mathbbm{1}_I(R_\alpha(t))T^{l_k(t)}(F)(\omega)-|I|\ \mathbb{E}\left[F\right]\right|
\end{align}

By Theorem \ref{teo1} the second summand is 0 on some full measure $\tilde{\Omega}$, and the first summand satisfies, for all $M$,
$$\limsup_N\frac{1}{N}\sum_{k=1}^NT^{l_k(t)}(|F_k-F|)(\omega)\leq\limsup_N\frac{1}{N}\sum_{k=1}^NT^{l_k(t)}|G_M|(\omega)$$
and since $G_M$ is bounded the average on the right converges to $\mathbb{E}[G_M]$ for all $\omega$ in a full measure set $\tilde{\Omega}_M$.

Taking $\Omega'=\tilde{\Omega}\cap\bigcap_M\tilde{\Omega}_M$ we get a full measure set such that for every $M$, every $\omega$ in $\Omega'$ and every $t\in[0,1)$
$$\limsup_N\left|\frac{1}{N}\sum_{k=1}^NT^{l_k(t)}(|F_k-F|)(\omega)\right|\leq\mathbb{E}[G_M]$$
Since $\mathbb{E}[G_M]\to 0$, this completes the proof.
\end{proof}
The following adapts the main result in \cite{Derri08} to our setting
\begin{theorem}\label{teo2}
Let $(\Omega,\mu,T)$ be weak mixing, let $\alpha\in[0,1)$ be irrational. Let $I\in\mathcal{I}$ and $F$, $G$ be functions in $L^\infty(\Omega)$.

Then, for each $t\in[0,1)$ there is a full measure $\Omega'\subseteq\Omega$ such that for all $\omega\in\Omega'$
\begin{align}\label{teo2eq}&\frac{1}{N}\sum_{k=1}^N\mathbbm{1}_I(R_\alpha^k(t)) T^{l_k(t)}(F)(\omega)T^k(G)(\omega)-\nonumber\\
&-\frac{1}{N}\sum_{k=1}^N \mathbbm{1}_I(R_\alpha^k(t))T^{l_k(t)}(\mathbb{E}[F\mid \mathcal{T}])(\omega)T^k(\mathbb{E}[G\mid \mathcal{T}])(\omega)\to 0
\end{align}
where $\mathcal{T}=\bigcap_k\mathcal{T}_k$ and $\mathcal{T}_k$ is the $\sigma$-algebra generated by the measurable functions $T^i(F),T^i(G)$ for $i\geq k$.
\end{theorem}
\begin{proof}Equation (\ref{teo2eq}) can be expanded telescopically using multilinearity, and is then equivalent to
\begin{align}\limsup_N&\left|\frac{1}{N}\sum_{k=1}^N\mathbbm{1}_I(R_\alpha^k(t)) T^{l_k(t)}(F-\mathbb{E}[F\mid \mathcal{T}])(\omega)T^k(G)(\omega)+\right.\nonumber\\
&\left.+\frac{1}{N}\sum_{k=1}^N \mathbbm{1}_I(R_\alpha^k(t))T^{l_k(t)}(\mathbb{E}[F\mid \mathcal{T}])(\omega)T^k(G-\mathbb{E}[G\mid \mathcal{T}])(\omega)\right|= 0
\end{align}

So it suffices to show
\begin{equation}\label{teo2eq2}\limsup_N\left|\frac{1}{N}\sum_{k=1}^N\mathbbm{1}_I(R_\alpha^k(t)) T^{l_k(t)}(F-\mathbb{E}[F\mid \mathcal{T}])(\omega)T^k(G)(\omega)\right|=0\end{equation}
and
\begin{equation}\label{teo2eq3}\limsup_N\left|\frac{1}{N}\sum_{k=1}^N \mathbbm{1}_I(R_\alpha^k(t))T^{l_k(t)}(\mathbb{E}[F\mid \mathcal{T}])(\omega)T^k(G-\mathbb{E}[G\mid \mathcal{T}])(\omega)\right|= 0
\end{equation}

Write $M$ for some common bound of $F$ and $G$. Notice $\mathcal{T}=\bigcap_k\mathcal{T}_k$ is the limit of a monotone sequence of $\sigma$-algebras, and is also $T$-invariant. By Doob's theorem, $\mathbb{E}[F\mid\mathcal{T}_k]\to\mathbb{E}[F\mid\mathcal{T}]$ both pointwise and in $L^1$, and the same is true for $G$.

By Theorem \ref{teo1}, there is a full measure $\Omega'\subset\Omega$ such that for all $\omega\in\Omega'$ and all $t\in[0,1)$
\begin{align}\limsup_N&\left|\frac{1}{N}\sum_{k=1}^N \mathbbm{1}_I(R_\alpha^k(t))T^{l_k(t)}(F-\mathbb{E}[F\mid \mathcal{T}])(\omega)T^k(G)(\omega)\right|\leq\nonumber\\
&\limsup_N\left|\frac{1}{N}\sum_{k=1}^N \mathbbm{1}_I(R_\alpha^k(t))T^{l_k(t)}(F-\mathbb{E}[F\mid \mathcal{T}_n])(\omega)T^k(G)(\omega)\right|+\nonumber\\
&+\limsup_N\frac{1}{N}\sum_{k=1}^N T^{l_k(t)}(|\mathbb{E}[F\mid \mathcal{T}]-\mathbb{E}[F\mid \mathcal{T}_n]|)(\omega)T^k(|G|)(\omega)\leq\nonumber\\
&\leq \limsup_N\left|\frac{1}{N}\sum_{k=1}^N \mathbbm{1}_I(R_\alpha^k(t))T^{l_k(t)}(F-\mathbb{E}[F\mid \mathcal{T}_n])(\omega)T^k(G)(\omega)\right|+\nonumber\\
&+M\left\|E[F\mid\mathcal{T}]-E[F\mid\mathcal{T}_n]\right\|_{L^1}\nonumber
\end{align}
and we can always choose a big enough $n$ independently of $\omega$ so that the last term is arbitrarily small. Therefore, to prove (\ref{teo2eq2}) it will suffice to prove\begin{equation}\label{teo2eq4}\limsup_N\left|\frac{1}{N}\sum_{k=1}^N \mathbbm{1}_I(R_\alpha^k(t))T^{l_k(t)}(F-\mathbb{E}[F\mid \mathcal{T}_n])(\omega)T^k(G)(\omega)\right|=0
\end{equation}
almost surely and for all $n$.

Analogously, using Birkhoff's ergodic theorem instead of Theorem \ref{teo1}, we can show that to prove (\ref{teo2eq3}) it suffices to prove
\begin{equation}\label{teo2eq5}\limsup_N\left|\frac{1}{N}\sum_{k=1}^N \mathbbm{1}_I(R_\alpha^k(t))T^{l_k(t)}(\mathbb{E}[F\mid \mathcal{T}])(\omega)T^k(G-\mathbb{E}[G\mid \mathcal{T}_n])(\omega)\right|= 0
\end{equation}
almost surely and for all $n$.

To prove both limits we will use Lyon's Law of Large Numbers for the following random variables (which depend on $t$, and hence the full measure set will no longer be independent of $t$)
$$X_k=\mathbbm{1}_I(R_\alpha^k(t))T^{l_k(t)}(F-\mathbb{E}[F\mid \mathcal{T}_n])T^k(G)$$
and
$$Y_k=\mathbbm{1}_I(R_\alpha^k(t))T^{l_k(t)}(\mathbb{E}[F\mid \mathcal{T}])T^k(G-\mathbb{E}[G\mid \mathcal{T}_n])$$

\begin{theorem}[Lyon's Law of Large Numbers, \cite{Lyons88}]\label{teolyons}Let $X_k$ be bounded, zero mean random variables such that
$$\sum_{N=0}^\infty\frac{1}{N}\left\|\frac{1}{N}\sum_{k=1}^NX_k\right\|_2^2<\infty$$
Then $(1/N)\sum_{k=1}^NX_k\to 0$ as $N\to\infty$ almost surely.
\end{theorem}

To see that $X_k$ and $Y_k$ satisfy the hypotheses of this LLN, notice first that, by invariance of $\mu$, for any measurable $f$ and $\sigma$-algebra $\mathcal{F}$ we have $T^k(\mathbb{E}[f\mid\mathcal{F}])=\mathbb{E}[T^k(f)\mid\ T^{-k}(\mathcal{F})]$, and hence, for any $f$ (recall (\ref{shiftedcond})),
\begin{equation}\label{einvariante}
T^k\left(\mathbb{E}[f\mid\mathcal{T}_n]\right)=\mathbb{E}[T^k(f)\mid\mathcal{T}_{n+k}]
\end{equation}

Let us show $X_k$ has zero mean for large enough $k$. Define $k_0=\lfloor n/(1-\alpha)\rfloor+1$ and by (\ref{einvariante})
\begin{align}
\mathbb{E}\left[X_k\right]&=\mathbbm{1}_I(R_\alpha^k(t))\mathbb{E}\left[\mathbb{E}[(T^{l_k(t)}(F)-\mathbb{E}[T^{l_k(t)}(F)\mid \mathcal{T}_{n+l_k(t)}])T^k(G)\mid \mathcal{T}_{n+l_k(t)}]\right]\nonumber\\
&=\mathbbm{1}_I(R_\alpha^k(t))\mathbb{E}\left[T^k(G)(\mathbb{E}[T^{l_k(t)}(F)\mid \mathcal{T}_{n+l_k(t)}]-\mathbb{E}\left[T^{l_k(t)}(F)\mid \mathcal{T}_{n+l_k(t)}\right])\right]\nonumber\\
&=0\nonumber
\end{align}
since $k>k_0$ implies $k>n+l_k(t)$ (and then $\mathcal{T}_k\subseteq\mathcal{T}_{n+l_k(t)}$, which implies $T^k(G)$ is $\mathcal{T}_{n+l_k(t)}$-measurable). Similarly,
\begin{align}
\mathbb{E}\left[Y_k\right]&=\mathbbm{1}_I(R_\alpha^k(t))\mathbb{E}\left[\mathbb{E}[\mathbb{E}[T^{l_k(t)}(F)\mid \mathcal{T}](T^k(G)-\mathbb{E}[T^k(G)\mid\mathcal{T}_{n+k}])\mid\ \mathcal{T}_{n+k}]\right]\nonumber\\
&=\mathbbm{1}_I(R_\alpha^k(t))\mathbb{E}\left[\mathbb{E}[T^{l_k(t)}(F)\mid \mathcal{T}](\mathbb{E}[T^k(G)\mid\ \mathcal{T}_{n+k}]-\mathbb{E}[T^k(G)\mid\mathcal{T}_{n+k}])\right]\nonumber\\
&=0\nonumber
\end{align}

Finally, let us verify the summability condition in Theorem \ref{teolyons}.

Consider $k,k'$ such that $k>k_0$ and $k'>k+\lfloor n/\alpha\rfloor+1$. We have that $l_{k'}(t)>l_k(t)+n$, so that $T^{l_{k'}(t)}(\mathbb{E}[F\mid\mathcal{T}_n])=\mathbb{E}[T^{l_{k'}(t)}(F)\mid\mathcal{T}_{n+l_{k'}(t)}]$ is $\mathcal{T}_{l_k(t)+n}$-measurable and hence $X_{k'}$ is $\mathcal{T}_{l_k(t)+n}$-measurable. Then,

$$\mathbb{E}\left[X_kX_{k'}\right]=\mathbbm{1}_I(R_\alpha^k(t))\mathbbm{1}_I(R_\alpha^{k'}(t))\mathbb{E}\left[X_{k'}\mathbb{E}[X_k\mid\mathcal{T}_{l_k(t)+n}]\right]$$
and given $k$ such that $R_\alpha^k(t)\in I$ (for otherwise $X_k$ is just 0)
\begin{align}
\mathbb{E}[X_k\mid\mathcal{T}_{l_k(t)+n}]&=T^k(G)\mathbb{E}[(T^{l_k(t)}(F)-\mathbb{E}[T^{l_k(t)}(F)\mid \mathcal{T}_{l_k(t)+n}])\mid \mathcal{T}_{l_{k}(t)+n}]=0\nonumber
\end{align}

Similarly, for any $k'>k+n$, $Y_{k'}$ is $\mathcal{T}_{k+n}$-measurable. Then
\begin{align}
\mathbb{E}\left[Y_kY_{k'}\right]=&\mathbbm{1}_I(R_\alpha^k(t))\mathbbm{1}_I(R_\alpha^{k'}(t))\mathbb{E}\left[Y_{k'}\mathbb{E}[Y_k\mid \mathcal{T}_{n+k}]\right]\nonumber
\end{align}
and given $k$ such that $R_\alpha^k(t)\in I$ (for otherwise $Y_k$ is just 0)
\begin{align}\mathbb{E}[Y_k\mid \mathcal{T}_{n+k}]&=\mathbb{E}[T^{l_k(t)}(F)\mid \mathcal{T}]\mathbb{E}\left[(T^k(G)-\mathbb{E}[T^k(G)\mid\mathcal{T}_{n+k}])\mid\mathcal{T}_{n+k}\right]=0\nonumber
\end{align}

That is, for $k>k_0$ and $k'>k+\lfloor n/\alpha\rfloor+1$ (resp. $k'>k+n$) the correlation of $X_k$ and $X_{k'}$ (resp. $Y_k$ and $Y_{k'}$) vanishes.

Since $X_k,Y_k\leq 2M^2$, we get
\begin{align}
&\left\|\frac{1}{N}\sum_{k=1}^NX_k\right\|_2^2\leq\frac{2}{N^2}\left[M^2N+\sum_{k_0<k<k'\leq N}^N\mathbb{E}[X_kX_{k'}]+\sum_{\text{rest of }k<k'\leq N}\mathbb{E}[X_kX_{k'}]\right]\leq\nonumber\\
&\leq\frac{2}{N^2}\left[M^2N+\sum_{k_0<k}^N\#\{k':k'\leq k+\lfloor n/\alpha\rfloor+1\}+4k_0NM^2\}\right]\leq\nonumber\\
&\leq\frac{2}{N^2}\left[M^2N+4k_0NM^2+\sum_{k_0<k}^N(\lfloor n/\alpha\rfloor+1)\right]<\frac{K}{N}<\infty\nonumber
\end{align}

where $K=2(M^2(4k_0+1)+\lfloor n/\alpha\rfloor+1)$ is independent of $N$, and we conclude

$$\sum_{N=0}^\infty\frac{1}{N}\left\|\frac{1}{N}\sum_{k=1}^NX_k\right\|_2^2\leq \sum_{N=0}^\infty\frac{K}{N^2}<\infty$$

The condition for $Y_k$ is established similarly.
\end{proof}

\begin{corollary}\label{coroteo2}Let $(\Omega,\mu,T)$, $I$, $F$ and $G$ be as in Theorem \ref{teo2} and suppose, in addition, that the system is Kolmogorov or exact.
Then, for each $t\in[0,1)$ there is a full measure $\Omega'\subseteq\Omega$ such that for all $\omega\in\Omega'$
\begin{align}\label{coro2eq}&\frac{1}{N}\sum_{k=1}^N\mathbbm{1}_I(R_\alpha^k(t)) T^{l_k(t)}(F)(\omega)T^k(G)(\omega)\to|I|\mathbb{E}[F]\mathbb{E}[G]
\end{align}
\end{corollary}
\begin{proof}Notice first that, by Theorem \ref{teo2}, (\ref{coro2eq}) is true when $F$ and $G$ are simple functions, since then $\mathcal{T}$ would be the tail event algebra of some finite partition, and such algebras are trivial when the system is K or exact.

Now, simple functions are dense in $L^\infty$, so for each $n$ we have a simple $F_n$ such that $\|F-F_n\|_\infty<2^{-n}$. And then, if $G$ is simple, using Theorem \ref{teo1} and the boundedness of $G$ we get, for every $\omega$ on a full measure $\Omega_n$

\begin{align}
&\limsup_N\left|\frac{1}{N}\sum_{k=1}^N\mathbbm{1}_I(R_\alpha^k(t))T^{l_k(t)}(F)(\omega)T^k(G)(\omega)-|I|\mathbb{E}[F]\mathbb{E}[G]\right|\leq\nonumber\\
&\leq\limsup_N\left|\frac{1}{N}\sum_{k=1}^N\mathbbm{1}_I(R_\alpha^k(t)) T^{l_k(t)}(F_n)(\omega)T^k(G)(\omega)-|I|\mathbb{E}[F]\mathbb{E}[G]\right|\leq\nonumber\\
&|I|\mathbb{E}[|F-F_n|]\mathbb{E}[G]<|I|2^{-n}\mathbb{E}[G]\nonumber
\end{align}
Which implies (\ref{coro2eq}) for every $F$ in $L^\infty$, every simple $G$ and every $\omega$ in the full measure set $\Omega'=\bigcap_n\Omega_n$.

We then approximate any bounded $G$ by simple functions $G_n$ and prove convergence in exactly the same way.
\end{proof}

Before ending this section, we recall another result we will use, a rather trivial extension of Doob's convergence theorem (a proof can be found in \cite[Theorem 2]{Blackwell61})

\begin{theorem}\label{teomart}Let $(f_k)_{k\in\NN}$ be a sequence of functions bounded in absolute value by some $g\in L^1$ such that $f_k\to f$ a.s. and let $\mathcal{F}_k$ be a monotone sequence of $\sigma$-algebras such that $\mathcal{F}_k\nearrow \mathcal{F}$ or $\mathcal{F}_k\searrow \mathcal{F}$. Then
$$\mathbb{E}\left[f_k\mid\ \mathcal{F}_k\right]\to \mathbb{E}\left[f\mid\ \mathcal{F}\right]\quad\text{a.s.}$$
\end{theorem}

\section{Gibbs measures}\label{secgibbs}In this section we introduce Gibbs measures and their properties which are most relevant for our needs.

Throughout this section $\Sigma$ will denote any finite alphabet (not necessarily $\Lambda\times\Lambda'$) and $Y\subseteq \Sigma^\NN$ will be a topologically transitive subshift of finite type (see \cite{Baladi2000} for definitions). For any $a\in\Sigma^*$ we will denote the cylinder of sequences in $X$ that have $a$ as a prefix by $[a]$ (note this set may be empty). If $\omega\in\Sigma^\NN$ and $a\in\Sigma^*$ we will denote by $a\omega$ the sequence that arises when concatening $a$ as a prefix of $\omega$.

For some fixed $0<\rho<1$ we endow $\Sigma^\NN$ with the metric $d(\omega,\tilde{\omega})=\rho^{\min\{n:\ \omega_n\neq\tilde{\omega}_n)\}}$. The topology induced by this metric is the product topology, which makes $\Sigma^\NN$ (and hence $Y$) a compact space. Let $C(Y)$ denote the space of real-valued continuous functions on $Y$. A function $\phi\in C(Y)$ is called H\"{o}lder if there is some constant $K>0$ such that $|\phi(\omega)-\phi(\tilde{\omega})|\leq Kd(\omega,\tilde{\omega})$. For any H\"{o}lder $\phi\in C(Y)$ define the \textit{transfer operator} on $C(Y)$ by
$$L_\phi(f)(\omega)=\sum_{\substack{a\in\Sigma\\a\omega\in X}}f(a\omega)e^{\phi(a\omega)}$$

As usual, $L_\phi^*$ will denote the dual operator on $C(Y)^*$, the space of regular Borel measures on $Y$, i.e., $L_\phi^*(\mu)(f)=\mu(L_\phi(f))$.

Write $S_m^\phi(\omega)=\sum_{j=0}^{m-1}\phi(T^j(\omega))$ and $E_m^\phi(\omega)= e^{S_{m}^\phi(\omega)}$. Observe that $E$ satisfies the cocycle condition $E_{|a|+m}^\phi(a\omega)=E_{|a|}^\phi(a\omega)E_m^\phi(\omega)$. The following observation follows from the definition of the transfer operator.

\begin{observation}$L^m_\phi(f)(\omega)=\sum_{\substack{a\in\Sigma^m\\a\omega\in X}}f(a\omega)E_m^\phi(\omega)$
\end{observation}

The following important theorem is the starting point for the theory of Gibbs measures (for a proof, see \cite[Theorem 1.5]{Baladi2000}).

\begin{theorem}[Ruelle-Perron-Frobenius]\label{teoruelle}Let $Y\subseteq\Sigma^\NN$ be a topologically transitive subshift of finite type and $\phi\in C(Y)$ be a H\"{o}lder function. There exist a real number $P$, a regular Borel measure $\nu$ with full support in $Y$ and a strictly positive function $\psi\in C(Y)$ such that:
\begin{itemize}
\item $L_\phi(\psi)=e^P\psi$
\item $L_\phi^*(\nu)=e^P\nu$
\item The measure $\mu$ defined by $\mu(A)=\int_Y\psi d\nu$ is shift-invariant.
\end{itemize}
\end{theorem}

Moreover, both $\mu$ and $\nu$ satisfy the following property.

\begin{definition}\label{defigibs}A measure $m$ supported on $Y$ is said to have the \textit{Gibbs property} if there is a function $\phi'$ (called \textit{potential}) and a positive constant $P'$ (called \textit{pressure}) such that for some constant $K>0$, all cylinders $[a]$ and all $\omega'=a\omega\in Y$, the following holds:
$$K^{-1}\leq\frac{m([a])}{e^{-|a|P'}e^{S_{|a|}^{\phi'}(\omega')}}\leq K$$
\end{definition}
\begin{lemma}\label{propgibs}The measures $\mu$ and $\nu$ provided by Theorem \ref{teoruelle} have the Gibbs property for the potential $\phi$ and pressure $P$. Moreover, $\mu$ is the only shift-invariant measure supported on $Y$ satisfying this property.
\end{lemma}

\begin{observation}\label{obsegibbs}There is a constant $K>0$ such that for all $a,b\in\Sigma^*$ and all $\omega,\tilde{\omega}\in X$ such that $ab\omega\in Y$, $ab\tilde{\omega}\in Y$, the following inequality holds
$$\left|\frac{E_{|a|}^\phi(ab\omega)}{E_{|a|}^\phi(ab\tilde{\omega})}-1\right|<K\rho^{|b|}$$
\end{observation}
\begin{proof}By the mean value theorem
\begin{align}
\left|\frac{E_{|a|}^\phi(ab\omega)}{E_{|a|}^\phi(ab\tilde{\omega})}-1\right|=\frac{e^\xi\left|S_{|a|}^\phi(ab\omega)-S_{|a|}^\phi(ab\tilde{\omega})\right|}{E_{|a|}^\phi(ab\tilde{\omega})}\nonumber
\end{align}
for some $\xi\in[\min(S_{|a|}^\phi(ab\omega),S_{|a|}^\phi(ab\tilde{\omega})),\max(S_{|a|}^\phi(ab\omega),S_{|a|}^\phi(ab\tilde{\omega}))]$. And since $\phi$ is H\"{o}lder we have
\begin{align}
\left|S_{|a|}^\phi(ab\omega)-S_{|a|}^\phi(ab\tilde{\omega})\right|&\leq\sum_{j=1}^{|a|}\left|\phi(T^j(ab\omega))-\phi(T^j(ab\tilde{\omega}))\right|\leq\nonumber\\
&\leq\sum_{j=0}^{|a|}K'\rho^{|a|-j}\rho^{|b|}=K\rho^{|b|}\sum_{j=0}^{|a|}\rho^{j}<K'\rho^{|b|}\frac{1}{1-\rho}\nonumber
\end{align}

Then, for $\xi'=\max(S_{|a|}^\phi(ab\omega),S_{|a|}^\phi(ab\tilde{\omega}))$ and for the constant $K$ provided by Lemma \ref{propgibs} we have
\begin{align}
\left|\frac{E_{|a|}^\phi(ab\omega)}{E_{|a|}^\phi(ab\tilde{\omega})}-1\right|&\leq \frac{e^{\xi'}K'\rho^{|b|}(1-\rho)^{-1}}{E_{|a|}^\phi(ab\tilde{\omega})}\leq\nonumber\\
&\leq\frac{K^2\mu([a])K'\rho^{|b|}(1-\rho)^{-1}}{\mu([a])}=K''\rho^{|b|}
\end{align}

\end{proof}

The following is a well-known property of Gibbs measures, and it is mentioned without a proof in \cite{Lalley86}. For the sake of completeness, we provide a proof.

Write $\hat{Y}$ for the two-sided subshift of finite type that extends $Y$, and $\mathcal{H}_k^l$ ($k,l\in\ZZ$) for the $\sigma$-algebra generated by the measurable functions $\omega\mapsto\omega_i$ for $k\leq i\leq l$.
\begin{lemma}\label{lemagibs}Let $\mu$ be the invariant Gibbs measure on $Y$ for the H\"{o}lder potential $\phi$ and let $\hat{\mu}$ be its bilateral extension to $\hat{Y}$. Then
$$\gamma_m= \sup \left\{\left|\frac{\hat{\mu}(A\mid \mathcal{H}_{-\infty}^0)(\omega)}{\hat{\mu}(A\mid \mathcal{H}_{-\infty}^0)(\tilde{\omega})}-1\right|: A\in\mathcal{H}_1^r,\ r<\infty;\omega_n=\tilde{\omega}_n\forall n,-m\leq n\leq 0\right\}$$
satisfies $\gamma_m<C\rho^m$.
\end{lemma}
\begin{proof}We first make the following observation. Let $\nu$ be the eigenvector measure of $L_\phi^*$ given by Theorem \ref{teoruelle} and $P$ be the pressure of $\phi$. For any $a=a_1\dots a_m\in\Sigma^*$ we have
\begin{align}\nu([a])=e^{-mP}L_\phi^{*m}(\nu)(\alpha)&=e^{-mP}\int\sum_{\substack{b\omega\\b\in\Sigma^m}}\mathbbm{1}_{[a]}(b\omega)e^{S_m^\phi(b\omega)} d\nu(\omega)=\nonumber\\
&=e^{-mP}\int e^{S_m^\phi(a\omega)} d\nu(\omega)\nonumber
\end{align}
Let $\psi$ be the eigenfunction of $L_\phi$. Since $d\mu=\psi d\nu$, we have
$$\mu([a])=e^{-mP}\int e^{S_m^\phi(a\omega)} \psi(\omega)d\nu(\omega)=e^{-|a|P}\int E_{|a|}^\phi(a\omega) \psi(\omega)d\nu(\omega)$$
Now take any $b,c\in\Sigma^*$ such that $[abc]$ is not empty (equivalently, $\nu([abc])>0$). Fix $p=E^\phi_{|a|}(abc\tilde{\omega})$ for some $\tilde{\omega}$ such that $abc\tilde{\omega}\in Y$.
\begin{align}
\frac{\mu([abc])}{\mu([ab])}&=\frac{e^{-|bc|P}\int \left(E_{|a|}^\phi(abc\omega)-p\right)E_{|bc|}^\phi(bc\omega)\psi(\omega)d\nu(\omega)+p\mu([bc])}{e^{-|b|P}\int \left(E_{|a|}^\phi(ab\omega)-p\right)E_{|b|}^\phi(b\omega)\psi(\omega)d\nu(\omega)+p\mu([b])}=\nonumber\\
&=\frac{\mu([bc])}{\mu([b])}\frac{M(a,bc,p)+1}{M(a,b,p)+1}
\end{align}
where
$$M(a,b,p)=e^{-|b|P}\frac{\int\left(E_{|a|}^\phi(ab\omega)/p-1\right) E_{|b|}^\phi(b\omega) \psi(\omega)d\nu(\omega)}{\mu([b])}$$

From Observation \ref{obsegibbs} and the fact that both $E_m^\phi$ and $\psi$ are nonnegative we get
$$\left|M(a,b,p)\right|\leq K\rho^{|b|}$$
for some constant $K$ independent of $a$, $b$ or $p$ (and similarly for $bc$ instead of $b$).

If we now take some $a'\in\Sigma^*$ and $p'=e^{S^\phi_{|a'|}(a'bc\tilde{\omega})}$ we get
$$\frac{\mu([abc])}{\mu([ab])}\frac{\mu([a'b])}{\mu([a'bc])}=\frac{N(a,a',bc,b,p,p')+1}{N(a',a,b,bc,p,p')+1}$$
where
$$N(a,a',d,d',p,p')=M(a,d,p)M(a',d',p')+M(a,d,p)+M(a',d',p')$$
and then $|N(a,a',d,d',p,p')|<K'\rho^{|b|}$ for some constant $K'$ independent of $a$, $a'$, $d$, $d'$, $p$ or $p'$. It follows that
\begin{equation}\label{ineqgibs}\left|\frac{\mu([abc]\mid[ab])}{\mu([a'bc]\mid[a'b])}-1\right|<K''\rho^{|b|}
\end{equation}
for some constant $K''$ independent of $a$, $a'$, $b$, $c$, $p$ and $p'$.

Now, a set $A\in\mathcal{H}_1^r$ with $r<\infty$ is actually a union of cylinders $[c]$ of length $r$, and since $\hat{\mu}(A\mid\mathcal{H}_{-\infty}^0)$ is the sum of $\hat{\mu}([c]\mid\mathcal{H}_{-\infty}^0)$ for all those finitely many $c$, it suffices to consider $A=[c]$ . By invariance, for any two-sided sequence $\omega$ in $\hat{Y}$ the conditional probability of $A=[c]$ on the past of length $k$ is $$\hat{\mu}(A\mid\mathcal{H}_{-k}^0)(\omega)=\mu([w_{-k}\dots w_0 c]\mid[w_{-k}\dots w_0])$$
So (\ref{ineqgibs}) implies that
\begin{equation}\label{ecugibs}
\left|\frac{\mu([w_{-k-l}\dots w_{-k-1}w_{-k}\dots w_0c]\mid[w_{-k-l}\dots w_{-k-1}w_{-k}\dots w_0])}{\mu([\tilde{w}_{-k-l}\dots \tilde{w}_{-k-1}w_{-k}\dots w_0c]\mid[\tilde{w}_{-k-l}\dots \tilde{w}_{-k-1}w_{-k}\dots w_0])}-1\right|<K''\rho^{|b|}
\end{equation}
and that
$$\left|\hat{\mu}(A\mid\mathcal{H}_{-k}^0)(\omega)-\hat{\mu}(A\mid\mathcal{H}_{-m}^0)(\omega)\right|<K''\rho^{\min(k,m)}$$

and hence $\hat{\mu}(A\mid\mathcal{H}_{-k}^0)(\omega)\to\hat{\mu}(A\mid\mathcal{H}_{-\infty}^0)(\omega)$ uniformly for all $\omega$ and all cylinder sets $A$. Combining this with (\ref{ecugibs}) we get the desired exponential speed of convergence.
\end{proof}

\section{Proof of the main result}\label{secproof}
In this section we will prove Theorem \ref{genericity2}. The outline of this proof is the following:

\begin{enumerate}
\item First of all, we reduce the problem to a fixed subalgebra $\mathcal{A}$ of $C(X)$ such that any $f$ in $\mathcal{A}$ is the linear combination of products of some characteristic functions of subsets of $\Omega_+$ times products of functionals $\phi_{[a]\times[b]}$ that evaluate the $\mathcal{P}(\Omega_+)$-component of $X$ on cylinder products of the form $[a]\times[b]\subseteq\Omega_+$. The $S$-shifts of $f$ will then have these functionals evaluated on magnified versions of the original measure, a magnification which amounts to a shift of the product cylinders and a conditioning of the measure .
\item In the conformal case (see \cite{HochmanJEMS}) this conditioning is equivalent via invariance to a conditioning on a monotone sequence of $\sigma$-algebras in the natural extension of the system to a $\mathbb{Z}$-shift, so that they may be treated using Doob's theorem and Maker's generalized ergodic theorem. The problem in our case is that cylinder products and the algebras on which conditioning occurs are shifted at different speeds ($k$ and $l_k(t)$, respectively), so this cannot be done so easily. Our first task is then to `split' those conditioned measures into two components, one shifted at speed $k$, the other at speed $l_k(t)$.
\item We will then show that these components converge almost surely to some fixed function on $\Omega$. For some of these components to converge an ``asymptotic memorylessness'' or ``asymptotic Markovianity'' will be needed. It is here that the Gibbs condition is used. 
\item Finally, our original averages will have been reduced to some multiple ergodic average of two functions, one shifted at speed $k$, the other at $l_k(t)$, and Theorem \ref{teo2} is used to show convergence. Ergodicity is proven in a similar fashion.
\end{enumerate}

The proof will assume $q=1$, even though Theorem \ref{genericity2} ensures the convergence of all $q$-sparse distributions. Yet convergence for $q>1$ is easily implied by the following facts:
\begin{enumerate}
\item The $q$-length translation $Y^q$ of a subshift of finite type $Y\subseteq \Sigma^\NN$ that regards each sequence in $Y$ as a sequence in $\left(\Sigma^q\right)^\NN$ (i.e. as a sequence of symbols that are words of length $q$) is a subshift of finite type for the shift transformation $T_q=T^q$.
\item A $T$-invariant measure $\mu$ supported on $Y$ is $T^q$-invariant and can thus be regarded as a $T_q$-invariant measure $\mu_q$ supported on $Y^q$.
\item If $\mu$ has the Gibbs property for a H\"{o}lder potential $\phi$ with pressure $P$ then the measure $\mu_q$ has the Gibbs property for the H\"{o}lder potential $\phi_q=S_q^\phi$ with pressure $qP$ (this can be checked directly from Definition \ref{defigibs}).
\item Then by Lemma \ref{propgibs} $\mu_q$ is the invariant measure provided by Theorem \ref{teoruelle}, so that Lemma \ref{lemagibs} applies and the following proof carries on for the system $(Y^q,\mu_q,T_q)$.
\end{enumerate}

\subsection{From conditional measures to ergodic averages}
Let $\mu$ be any shift-invariant, weak mixing measure. We recall some notation of Subsection \ref{subsecmeasurescp}: $\phi_{\overline{C}}(\nu)=\nu(\overline{C})$ and also $\overline{A}=\pi_\Lambda^{-1}(A)\times\Lambda'^\ZZ$ for $A$ a measurable subset of $\Lambda^\NN$. When there is no additional subscript, $\mathbb{E}$ will denote expectation with respect to the fixed, original measure $\mu$.

Recall the definition of $\mu_{t,k}^\omega$ and notice that, if we extend $\mu$ to the $\ZZ$-system $\Omega$ and we take any $\tilde{\omega}\in\pi^{-1}(\omega)$, we have, by (\ref{shiftedcond}),

\begin{align}\mu_{t,k}^{\omega}(A\times B)&=\mu\left(T^{-k}(\overline{A})\cap T^{-l_k(t)}(\overline{B})\mid \ \mathcal{F}_1^k\vee\mathcal{G}_1^{l_k(t)}\right)(\omega)=\nonumber\\
&=\mathbb{E}\left[T^{k}(\mathbbm{1}_{\overline{A}} ) T^{l_k(t)}(\mathbbm{1}_{\overline{B}})\mid \mathcal{F}_1^k\vee\mathcal{G}_1^{l_k(t)}\right](\tilde{\omega})=\nonumber\\
&=\mathbb{E}\left[T^{k-r}(\mathbbm{1}_{\overline{A}} ) T^{l_k(t)-r}(\mathbbm{1}_{\overline{B}})\mid \mathcal{F}_{1-r}^{k-r}\vee\mathcal{G}_{1-r}^{l_k(t)-r}\right]\circ T^r(\tilde{\omega})=\nonumber\\
%&=\frac{\mu\left(T^{-k}(\overline{A})\cap T^{-l_k(t)}(\overline{B})\cap [\tilde{\omega}^1_1\dots\tilde{\omega}^1_k]\times[\tilde{\omega}^2_1\dots\tilde{\omega}^2_{l_k(t)}]\right)}{\mu\left([\tilde{\omega}^1_1\dots\tilde{\omega}^1_k]\times[\tilde{\omega}^2_1\dots\tilde{\omega}^2_{l_k(t)}]\right)}=\nonumber\\
%&=\frac{\mu\left(T^{-(k-r)}(\overline{A})\cap T^{-(l_k(t)-r)}(\overline{B})\cap T^r([\tilde{\omega}^1_1\dots\tilde{\omega}^1_k]\times[\tilde{\omega}^2_1\dots\tilde{\omega}^2_{l_k(t)}])\right)}{\mu\left(T^r([\tilde{\omega}^1_1\dots\tilde{\omega}^1_k]\times[\tilde{\omega}^2_1\dots\tilde{\omega}^2_{l_k(t)}])\right)}=\nonumber\\
%\footnotemark
&=\mu\left(T^{-(k-r)}(\overline{A})\cap T^{-(l_k(t)-r)}(\overline{B})\mid \ \mathcal{F}_{1-r}^{k-r}\vee\mathcal{G}_{1-r}^{{l_k(t)-r}}\right)(T^r(\tilde{\omega}))\nonumber
\end{align}

%\footnotetext{From a probabilistic point of view, notice here that $T^r([\tilde{\omega}^1_1\dots\tilde{\omega}^1_k]\times[\tilde{\omega}^2_1\dots\tilde{\omega}^2_{l_k(t)}])$ is just the event that the $\Lambda$-coordinates from $1-r$ to $k-r$ coincide with $\tilde{\omega}^1_{1}\dots\tilde{\omega}^1_k$ and the $\Lambda'$-coordinates from $1-r$ to $l_k(t)-r$ coincide with $\tilde{\omega}^2_{1}\dots\tilde{\omega}^2_{l_k(t)}$}

Thus, setting $r=l_k(t)$, when $\mu_{t,k}^\omega$ is evaluated on product sets it has the form
\begin{equation}\label{ecu}\mu_{t,k}^{\omega}(A\times B)=F_{t,k}^{A\times B}\circ T^{l_k(t)}(\iota(\omega))
\end{equation}
for any $\iota(\omega)\in\Omega$ such that $\pi(\iota(\omega))=\omega$ (i.e. $\iota$ is a section of $\pi$) and $F_{t,k}^{A\times B}:\Omega\to\RR$ the measurable function given by
\begin{align}
F_{t,k}^{A\times B}(\omega)&=\mu\left(T^{-s_k(t)}(\overline{A})\cap \overline{B}\mid \ \mathcal{F}_{-l_k(t)+1}^{s_k(t)}\vee\mathcal{G}_{-l_k(t)+1}^0\right)(\omega)\nonumber\\
&=\mathbb{E}\left[\mathbbm{1}_{T^{-s_k(t)}(\overline{A})}\mathbbm{1}_{\overline{B}}\mid \ \mathcal{F}_{-l_k(t)+1}^{s_k(t)}\vee\mathcal{G}_{-l_k(t)+1}^0\right](\omega)
\end{align}
with $s_k(t):=k-l_k(t)$ (as with $l_k(t)$ we may write $s_k=s_k(0)$).

Notice that the fact that (\ref{ecu}) is true for any $\iota(\omega)$ such that $\pi(\iota(\omega))=\omega$ means $F_{t,k}^{A\times B}\circ T^{l_k(t)}$ is $\mathcal{F}_1^\infty\vee\mathcal{G}_1^\infty$-measurable.

We will also use the following measurable functions on $\Omega$, for measurable $A\times B\subseteq \Sigma^\NN$.

$$G_{t,k}^{A\times B}(\omega)=\mu\left(\overline{A}\cap T^{-s_k(t)}(\overline{B})\mid \ \mathcal{F}_{-k+1}^0\vee\mathcal{G}_{-k+1}^{-s_k(t)}\right)(\omega)$$
and, again by invariance,

$$\mu_{t,k}^{\omega}(A\times B)=G_{t,k}^{A\times B}\circ T^{k}(\iota(\omega))$$
for any $\iota(\omega)\in\pi^{-1}(\omega)$.

Now, consider the following set of measurable functions on $X$.
\begin{align}
\mathcal{S}=&\{f(t,\omega,\nu)=\mathbbm{1}_I(t)\mathbbm{1}_{\overline{[c]}}(\omega)\mathbbm{1}_{\overline{[d]}}(\omega)\phi_{[a_1]\times [b_1]}(\nu)\dots\phi_{[a_r]\times [b_r]}(\nu):\nonumber\\
&I\in\mathcal{I},\ c,a_i\in \Lambda^*,\ d, b_i\in \Lambda'^*,\text{ for some $r$ and $i\leq r$}\}\nonumber
\end{align}
Let $\mathcal{A}$ be the linear span of $\mathcal{S}$. $\mathcal{A}$ clearly is an algebra that separates points and even though its functions are not continuous ($\mathbbm{1}_I$ is not continuous on $\mathbb{T}$), it is well known that any continuous $g\in C(\mathbb{T})$ can be uniformly approximated by functions in the linear span of $\{\mathbbm{1}_I:\ I\text{  an interval}\}$ and this implies that any $f\in C(X)$ may be uniformly approximated by functions in $\mathcal{A}$. Hence, it is enough to show the convergence of (\ref{genericity}) for functions in $\mathcal{S}$. 

Fix some $f\in\mathcal{S}$, that is,
\begin{equation}\label{efe}
f(t,\omega,\nu)=\mathbbm{1}_I(t)\mathbbm{1}_{[c]}(\omega^1)\mathbbm{1}_{[d]}(\omega^2)\phi_{[a_1]\times [b_1]}(\nu)\dots\phi_{[a_r]\times [b_r]}(\nu)
\end{equation}

As was hinted at the beginning of this section, we would like to somehow `split' the shifted conditional probabilities $S^k(\phi_{[a_i]\times[b_i]})(t,\omega,\mu)=\mu_{t,k}^{\omega}([a_i]\times [b_i])$ into the product of two conditional probabilities, one evaluating on a set shifted by $k$, the other one evaluating on a set shifted by $l_k(t)$.

Notice first that we have a family of orthogonal projections $\{\Pi_{m,n}'\}_{m,n\in\mathbb{N} }$ on $L^2(\Omega)$ that map $f\in L^2(\Omega)\mapsto \mathbb{E}[f\mid \mathcal{F}_{-n+1}^m\vee\mathcal{G}^{0}_{-n+1}]$. Thus, for any $i$, we may then write
$$\mathbbm{1}_{\overline{[b_i]}}=\Pi_{t,k}(\mathbbm{1}_{\overline{[b_i]}})+\Phi_{t,k}(\mathbbm{1}_{\overline{[b_i]}})$$
where $\Pi_{t,k}=\Pi_{s_k(t),l_k(t)}'$ and $\Phi_{t,k}=\textbf{id}-\Pi_{t,k}$ (similarly, we write $\Pi_\infty=\Pi_{\infty,\infty}'$ and $\Phi_\infty=\textbf{id}-\Pi_\infty$). Then,
\begin{align}\label{ecusplit}
F_{t,k}^{[a_i]\times[b_i]}&=\mathbb{E}\left[\mathbbm{1}_{\overline{[b_i]}}\mathbbm{1}_{T^{-s_k(t)}(\overline{[a_i]})} \mid \ \mathcal{F}_{-l_k(t)+1}^{s_k(t)}\vee\mathcal{G}_{-l_k(t)+1}^{0}\right]=\nonumber\\
&=\mathbb{E}\left[\Pi_{t,k}(\mathbbm{1}_{\overline{[b_i]}})\mathbbm{1}_{T^{-s_k(t)}(\overline{[a_i]})} \mid \ \mathcal{F}_{-l_k(t)+1}^{s_k(t)}\vee\mathcal{G}_{-l_k(t)+1}^{0}\right]+\nonumber\\
&+\mathbb{E}\left[\Phi_{t,k}(\mathbbm{1}_{\overline{[b_i]}})\mathbbm{1}_{T^{-s_k(t)}(\overline{[a_i]})} \mid \ \mathcal{F}_{-l_k(t)+1}^{s_k(t)}\vee\mathcal{G}_{-l_k(t)+1}^{0}\right]\nonumber\\
&=\Pi_{t,k}(\mathbbm{1}_{\overline{[b_i]}})\mathbb{E}\left[\mathbbm{1}_{T^{-s_k(t)}(\overline{[a_i]})} \mid \  \mathcal{F}_{-l_k(t)+1}^{s_k(t)}\vee\mathcal{G}_{-l_k(t)+1}^{0}\right]+\nonumber\\
&+\Pi_{t,k}\left(\Phi_{t,k}(\mathbbm{1}_{\overline{[b_i]}})\mathbbm{1}_{T^{-s_k(t)}(\overline{[a_i]})} \right)\nonumber=\\
&=\Pi_{t,k}(\mathbbm{1}_{\overline{[b_i]}})\Pi_{t,k}\left(\mathbbm{1}_{T^{-s_k(t)}(\overline{[a_i]})} \right)+\Pi_{t,k}\left(\Phi_{t,k}(\mathbbm{1}_{\overline{[b_i]}})\mathbbm{1}_{T^{-s_k(t)}(\overline{[a_i]})} \right)
\end{align}
The first summand, which we will denote by $C_{t,k}^{[a_i]\times[b_i]}$, is our desired `split' term. Denote the second summand by $D_{t,k}^{[a_i]\times[b_i]}$. Next, we will show this term can be ignored.

\begin{lemma}\label{lema1}$D_{t,k}^{[a_i]\times[b_i]}\to 0$ a.s. when $k\to \infty$
\end{lemma}

\begin{proof}Notice that $\Phi_{t,k}(\mathbbm{1}_{\overline{[b_i]}})=\mathbbm{1}_{\overline{[b_i]}}-\mathbb{E}\left[\mathbbm{1}_{\overline{[b_i]}}\mid \  \mathcal{F}_{-l_k(t)+1}^{s_k(t)}\vee\mathcal{G}_{-l_k(t)+1}^{0}\right]$. So
$$\Phi_{t,k}(\mathbbm{1}_{\overline{[b_i]}})\to \mathbbm{1}_{\overline{[b_i]}}-\mathbb{E}\left[\mathbbm{1}_{\overline{[b_i]}}\mid \  \mathcal{F}_{-\infty}^{\infty}\vee\mathcal{G}_{-\infty}^{0}\right]=\Phi_\infty(\mathbbm{1}_{\overline{[b_i]}})$$
a.s. and in $L^2$ by the martingale convergence theorem and the fact that $s_k(t),l_k(t)\to \infty$.

Then $\mathbbm{1}_{T^{-s_k(t)}(\overline{[a_i]})}(\Phi_\infty(\mathbbm{1}_{\overline{[b_i]}})-\Phi_{t,k}(\mathbbm{1}_{\overline{[b_i]}}))\to 0$ a.s., and if we write
\begin{align}
\tilde{D}_{t,k}^{[a_i]\times[b_i]}&=\mathbb{E}\left[\mathbbm{1}_{T^{-s_k(t)}(\overline{[a_i]})}\Phi_\infty(\mathbbm{1}_{\overline{[b_i]}})\mid \  \mathcal{F}_{-l_k(t)+1}^{s_k(t)}\vee\mathcal{G}_{-l_k(t)+1}^{0}\right]=\nonumber\\
&=\Pi_{t,k}\left(\mathbbm{1}_{T^{-s_k(t)}(\overline{[a_i]})}\Phi_\infty(\mathbbm{1}_{\overline{[b_i]}})\right)\nonumber
\end{align}
we have, by Theorem \ref{teomart},
\begin{align}D_{t,k}^{[a_i]\times[b_i]}-\tilde{D}_{t,k}^{[a_i]\times[b_i]}&=\mathbb{E}\left[\mathbbm{1}_{T^{-s_k(t)}(\overline{[a_i]})}(\Phi_{t,k}-\Phi_\infty)(\mathbbm{1}_{\overline{[b_i]}})\mid   \mathcal{F}_{-l_k(t)+1}^{s_k(t)}\vee\mathcal{G}_{-l_k(t)+1}^{0}\right]\nonumber
\end{align}
which converges a.s. to $\Pi_\infty(0)=0$. Let us now show that $\tilde{D}_{t,k}^{[a_i]\times[b_i]}=0$. Since conditional expectation is an orthogonal projection we have

\begin{align}
\left\|\tilde{D}_{t,k}^{[a_i]\times[b_i]}\right\|_2^2&=\int_\Omega\left|\mathbb{E}\left[\mathbbm{1}_{T^{-s_k(t)}(\overline{[a_i]})}\Phi_\infty(\mathbbm{1}_{\overline{[b_i]}})\mid \  \mathcal{F}_{-l_k(t)+1}^{s_k(t)}\vee\mathcal{G}_{-l_k(t)+1}^{0}\right]\right|^2\ d\mu=\nonumber\\
&=\int_\Omega\mathbbm{1}_{T^{-s_k(t)}(\overline{[a_i]})}\Phi_\infty(\mathbbm{1}_{\overline{[b_i]}})\ \tilde{D}_{t,k}^{[a_i]\times[b_i]}\ d\mu=\nonumber\\
&=\big<\Phi_\infty(\mathbbm{1}_{\overline{[b_i]}}),\mathbbm{1}_{T^{-s_k(t)}(\overline{[a_i]})}\Pi_{t,k}\left(\mathbbm{1}_{T^{-s_k(t)}(\overline{[a_i]})}\Phi_\infty(\mathbbm{1}_{\overline{[b_i]}})\right)\big>_{L^2(\Omega)}\nonumber
\end{align}
Notice that the $L^2$ function $\Phi_\infty(\mathbbm{1}_{\overline{[b_i]}})$ is orthogonal to any  $\mathcal{F}_{-\infty}^{\infty}\vee\mathcal{G}_{-\infty}^{0}$-measurable function, while at the same time $\Pi_{t,k}\left(\mathbbm{1}_{T^{-s_k(t)}(\overline{[a_i]})}\Phi_\infty(\mathbbm{1}_{\overline{[b_i]}})\right)$ and $\mathbbm{1}_{T^{-s_k(t)}(\overline{[a_i]})}$ are $\mathcal{F}_{-\infty}^{\infty}\vee\mathcal{G}_{-\infty}^{0}$-measurable, since $\Pi_{t,k}$ conditions on a subalgebra of $\mathcal{F}_{-\infty}^{\infty}\vee\mathcal{G}_{-\infty}^{0}$ and $T^{-s_k(t)}(\overline{[a_i]})=T^{-s_k(t)}([a_i]\times\Lambda'^\ZZ)\in\mathcal{F}_{-\infty}^\infty$ is measurable in the $\sigma$-algebra generated by the first coordinate.

Hence, $\left\|\tilde{D}_{t,k}^{[a_i]\times[b_i]}\right\|_2=0$ for all $k$ and $D_{t,k}^{[a_i]\times[b_i]}\to 0$ a.s.
\end{proof}

To simplify notation, we may write $\varphi^k(t,\omega)=\mathbbm{1}_I(R_\alpha^k(t))T^k(\mathbbm{1}_{\overline{[c]}})(\iota(\omega))$ for any $\iota(\omega)$ in $\Omega$ such that $\pi(\iota(\omega))=\omega$.

\begin{corollary}\label{coro1}There is a full measure set $\Omega'\subseteq\textnormal{supp}\ \mu\subseteq \Omega$ such that for all $\omega\in\pi(\Omega')$, all $t\in[0,1)$ and all $\iota(\omega)\in\Omega'$, we have

%\begin{align}\label{ecu2}&\frac{1}{N}\sum_{k=1}^N\left[f(S^k(t,\omega,\mu))-\varphi^k(t,\omega)T^{l_k(t)}(\mathbbm{1}_{[d]})(\omega^2)\vphantom{\prod_i^s}\right.\nonumber\\
%&\left.\quad\quad\prod_{i=1}^r T^{l_k(t)}\left(C_{t,k}^{[a_i]\times[b_i]}\right)(\iota(\omega))\right]\xrightarrow{} 0
%\end{align}

\begin{align}\label{ecu2}&\frac{1}{N}\sum_{k=1}^N\left[f(S^k(t,\omega,\mu))-\varphi^k(t,\omega)T^{l_k(t)}\left(\mathbbm{1}_{\overline{[d]}}\prod_{i=1}^r\left(C_{t,k}^{[a_i]\times[b_i]}\right)\right)(\iota(\omega))\right]\xrightarrow{} 0
\end{align}
\end{corollary}
\begin{proof}For any $\omega\in\text{supp}\mu$
$$f(S^k(t,\omega,\mu))=\mathbbm{1}_I(R_\alpha^k(t))T^k(\mathbbm{1}_{[c]})(\omega^1)T^{l_k(t)}(\mathbbm{1}_{[d]})(\omega^2)\prod_{i=1}^r \mu_{t,k}^{\omega}([a_i]\times[b_i])$$
so, recalling (\ref{ecu}) and Observation \ref{obseboba}, we have, 
\begin{align}
&\left|\frac{1}{N}\sum_{k=1}^N\left[f(S^k(t,\omega,\mu))-\varphi^k(t,\omega)T^{l_k(t)}\left(\mathbbm{1}_{\overline{[d]}}\prod_{i=1}^r\left(C_{t,k}^{[a_i]\times[b_i]}\right)\right)(\iota(\omega))\right]\right|\leq\nonumber\\
&\leq\frac{1}{N}\sum_{k=1}^N\sum_{j=1}^r\left|\varphi^k(t,\omega)T^{l_k(t)}(\mathbbm{1}_{[d]})(\omega^2)\prod_{i=1}^{j-1} \mu_{t,k}^{\omega}([a_i]\times[b_i])\right.\nonumber\\
&\left.\left[\mu_{t,k}^{\omega}([a_j]\times[b_j])- T^{l_k(t)}\left(C_{t,k}^{[a_j]\times[b_j]}\right)(\iota(\omega))\right]\prod_{i=j+1}^r  T^{l_k(t)}\left(C_{t,k}^{[a_i]\times[b_i]}\right)(\iota(\omega))\right|\leq\nonumber\\
&\leq\sum_{j=1}^r\frac{1}{N}\sum_{k=1}^N\left|\mu_{t,k}^{\omega}([a_j]\times[b_j])- T^{l_k(t)}\left(C_{t,k}^{[a_j]\times[b_j]}\right)(\iota(\omega))\right|=\nonumber\\
&\leq\sum_{j=1}^r\frac{1}{N}\sum_{k=1}^N\left| T^{l_k(t)}\left(F_{t,k}^{[a_j]\times[b_j]}-C_{t,k}^{[a_j]\times[b_j]}\right)(\iota(\omega))\right|=\nonumber\\
&=\sum_{j=1}^r\frac{1}{N}\sum_{k=1}^N\left| T^{l_k(t)}\left(D_{t,k}^{[a_j]\times[b_j]}\right)(\iota(\omega))\right|
\end{align}
By virtue of Corollary \ref{coro2} and Lemma \ref{lema1} there's a full measure $\Omega'\subseteq\text{supp}\ \mu\subseteq \Omega$ such that the last term converges to zero when $\iota(\omega)$ is in $\Omega'$, and since the last term depends only on $\pi(\iota(\omega))$, the same will be true for $\pi^{-1}(\pi(\Omega')) $, and any $\omega\in\pi(\Omega')$ satisfies (\ref{ecu2}).
\end{proof}

Now $C_{t,k}^{[a_i]\times[b_i]}=\Pi_{t,k}(\mathbbm{1}_{\overline{[b_i]}})\Pi_{t,k}\left(\mathbbm{1}_{T^{-s_k(t)}(\overline{[a_i]})} \right)$, so writing $f_{t,k}^i=\Pi_{t,k}(\mathbbm{1}_{\overline{[b_i]}})$ and $h_{t,k}^i=\Pi_{t,k}\left(\mathbbm{1}_{T^{-s_k(t)}(\overline{[a_i]})} \right)$ we then have

\begin{align}\label{eq2}
&\frac{1}{N}\sum_{k=1}^N\mathbbm{1}_I(R_\alpha^k(t))T^k(\mathbbm{1}_{[c]})(\omega^1)T^{l_k(t)}(\mathbbm{1}_{[d]})(\omega^2)\prod_{i=1}^r T^{l_k(t)}\left(C_{t,k}^{[a_i]\times[b_i]}\right) (\iota(\omega))=\nonumber\\
&=\frac{1}{N}\sum_{k=1}^N\varphi^k(t,\omega)T^{l_k(t)}\left(\mathbbm{1}_{\overline{[d]}}\prod_{i=1}^r(f_{t,k}^i)h_{t,k}^i\right)(\iota(\omega))
\end{align}

Write $f_i^*=\Pi_\infty(\mathbbm{1}_{\overline{[b_i]}})$. Then we have the following observation
\begin{observation}\label{obse2}There is a full measure $\Omega'\subseteq\text{supp }\mu\subseteq\Omega$ such that, for all $\omega\in\pi(\Omega')$, all $t\in[0,1)$ and all $\iota(\omega)\in\Omega'\cap \pi^{-1}(\{\omega\})$, the averages
\begin{equation}\label{eq3}\frac{1}{N}\sum_{k=1}^N\left[S^k(f)(t,\omega,\mu)-\varphi^k(t,\omega)T^{l_k(t)}\left(\mathbbm{1}_{\overline{[d]}}\prod_{i=1}^r(f_{i}^*)h_{t,k}^i\right)(\iota(\omega))\right]
\end{equation}
converge to 0.
\end{observation}
\begin{proof}We use multilinearity to decompose the average telescopically and get

\begin{align}\label{eculoca}
&\limsup_N\left|\frac{1}{N}\sum_{k=1}^N\left[S^k(f)(t,\omega,\mu)-\varphi^k(t,\omega)T^{l_k(t)}\left(\mathbbm{1}_{\overline{[d]}}\prod_{i=1}^r(f_{i}^*)h_{t,k}^i\right)(\iota(\omega))\right]\right|\leq\nonumber\\
&\limsup_N\left|\frac{1}{N}\sum_{k=1}^N\left[S^k(f)(t,\omega,\mu)-\varphi^k(t,\omega)T^{l_k(t)}\left(\mathbbm{1}_{\overline{[d]}}\prod_{i=1}^r(f_{t,k}^i)h_{t,k}^i\right)(\iota(\omega))\right]\right|+\nonumber\\
&\limsup_N\left|\sum_{j=1}^r\frac{1}{N}\sum_{k=1}^N\varphi^k(t,\omega)T^{l_k(t)}\left(\psi^{k,j}_t\right)(\iota(\omega))\right|
\end{align}
where $\psi^{k,j}_t=\mathbbm{1}_{\overline{[d]}}\left(\prod_{\substack{i=1}}^{j-1}f_{t,k}^ih_{t,k}^i\right)(f_{t,k}^j-f_j^*)h_{t,k}^j\left(\prod_{\substack{i=j+1}}^{r}f_{i}^*h_{t,k}^i\right)$. Since the first summand of (\ref{eculoca}) vanishes by (\ref{eq2}) and Corollary \ref{coro1}, we get that the whole of (\ref{eculoca}) is bounded by
$$\left|\sum_{j=1}^r\frac{1}{N}\sum_{k=1}^NT^{l_k(t)}|f_{t,k}^j-f_j^*|(\iota(\omega))\right|$$
To see that this average converges to 0 a.s. notice that $|f_{t,k}^i-f^*_i|\to 0$ a.s. by Doob's convergence theorem and then convergence for the average follows from Corollary \ref{coro2} for some full measure $\Omega''\subseteq\Omega$, and, by dependence of the last term only on $\pi(\iota(\omega))$, also for $\pi^{-1}(\pi(\Omega''))$. So  we deduce there is some $\Omega'\subseteq\text{supp}\ \mu\subseteq\Omega$ such that (\ref{eq3}) holds for all $\omega\in\pi(\Omega')$ and all $\iota(\omega)$.

\end{proof}

\subsection{Genericity}\label{secgenericity}
Notice that up to this point we have been working with quite general weak mixing measures, without invoking any Gibbs property. We explain the reason why we now need to do so. By this point our original averages have been reduced to averages of some shifted conditional expectations, where conditioning is on $\sigma$-algebras that are ``moving'', i.e., depend on $k$. The problem is that the conditional expectations $h^i_{t,k}$ are expectations of the characteristic function of a set that is not fixed (namely, $T^{-s_k(t)}(A_i)$), and if we try to shift it, then we do get conditional probability of a fixed set, but conditioning is on a nonmonotone sequence of $\sigma$-algebras (so that we can't use Doob's theorems). What we want to use, then, is some property that ensures that, even though the set is not fixed, $h_{t,k}^i$ will be uniformly close to the conditional probability with respect to some fixed, limit $\sigma$-algebra, which will be the whole past $\mathcal{F}_{-\infty}^0\vee\mathcal{G}_{-\infty}^0$. In other words, we want conditioning on the ``recent'' past to be uniformly close to conditioning on the whole past. This is achieved, of course, by Markov measures, which are a strict subset of Gibbs measures. Lemma \ref{lemagibs} could be read as saying that Gibbs measures are ``asymptotically and exponentially close'' to a Markov measure. In fact, we won't be using the exponential speed, it is only the uniformity of the convergence of the recent past to the whole past that will be essential, and the proof is valid for any measure that satisfies that uniformity. Since we do not know of any broader and widely used class of measures that satisfy that uniformity, results are stated for Gibbs measures only.

Finally, while the following theorem does present a closed expression for the limit of the averages when $f$ is in our dense algebra, it is evident that such an expression admits no natural generalization to any $f\in C(X)$ and thus the task of providing a closed expression for the generated CP-chain distribution $Q$ seems hopeless.

\begin{theorem}\label{teoconvergence}Suppose that $\mu$ is an invariant Gibbs measure for some H\"{o}lder potential supported on a topologically transitive subshift of finite type $Y\subseteq\Omega_+$. Let $f$ be as in (\ref{efe}). Then for each $t\in[0,1)$ there is a full measure set $\Omega'_+\subseteq \Omega_+$ such that the averages
$$\frac{1}{N}\sum_{k=1}^NS^k (f)(t,\omega,\mu)$$
converge for all $\omega\in\Omega'_+$ to $|I|\mathbb{E}[\tilde{f}]\mathbb{E}[\tilde{g}]$, where:
\begin{equation}\tilde{g}=\mathbbm{1}_{\overline{[c]}}\prod_{i=1}^r\mathbb{E}\left[\mathbbm{1}_{\overline{[a_i]}}\biggr\rvert\bigcap_{k=1}^\infty\mathcal{F}_{-\infty}^0\vee\mathcal{G}_{-\infty}^{-k}\right]\quad\text{ and }\quad\tilde{f}=\mathbbm{1}_{\overline{[d]}}\prod_{i=1}^r(f_{i}^*)\nonumber
\end{equation}

In particular, when $t=0$ $\mu$ generates a CP-chain distribution $Q$ on $X$.

\end{theorem}
\begin{proof}It suffices to prove convergence of the averages

\begin{equation}\label{eqav}
\frac{1}{N}\sum_{k=1}^N\left[\mathbbm{1}_I(R_\alpha^k(t))T^k(\mathbbm{1}_{\overline{[c]}})T^{l_k(t)}\left(\mathbbm{1}_{\overline{[d]}}\prod_{i=1}^r(f_{i}^*)h_{t,k}^i\right)(\omega')\right]
\end{equation}
for a.e. $\omega'\in\hat{Y}$ ($\hat{Y}$ is the natural extension of $Y$\footnote{The natural extension of $Y$ is the support of the natural extension of the Gibbs measure or, equivalently, the bilateral subshift defined by the same local rules that define $Y$.}), since then we would have, by Observation \ref{obse2}, a full measure $\Omega'\subseteq\Omega$ such that $\frac{1}{N}\sum_{k=1}^NS^k (f)(t,\omega,\mu)$ converges for all $\omega\in\pi(\Omega')$. And since by definition of the natural extension of $\mu$ from $\Omega_+$ to $\Omega$ we have $\mu(A)=\mu(\pi^{-1}(A))$  for any measurable $A\subseteq\Omega_+$, then $$\mu(\pi(\Omega'))=\mu(\pi^{-1}(\pi(\Omega')))\geq\mu(\Omega')=1$$
and we may define $\Omega_+'=\pi(\Omega')$.

So let us write $\tilde{g}_{t,k}^i=\mathbb{E}[\mathbbm{1}_{\overline{[a_i]}}\mid\ \mathcal{F}_{-\infty}^{0}\vee\mathcal{G}_{-\infty}^{-s_k(t)}]$, $\tilde{h}_{t,k}^i= T^{s_k(t)}(\tilde{g}_{t,k}^i)$ and
$$z_{t,k}^i(\omega)=\mu(T^{-s_k(t)}([a_i])\cap([\omega_1^1\dots\omega_{s_k(t)}^1]\times\Lambda'^{\ZZ})\mid \mathcal{F}_{-\infty}^{0}\vee\mathcal{G}_{-\infty}^{0})(\omega)$$
$$z_{t,k}(\omega)=\mu([\omega_1^1\dots\omega_{s_k(t)}^1]\times\Lambda'^{\ZZ}\mid \mathcal{F}_{-\infty}^{0}\vee\mathcal{G}_{-\infty}^{0})(\omega)$$

We will use the following fact. For any finite measurable partition $\alpha$ write, by abuse of notation, $\alpha$ for the $\sigma$-algebra generated, and also $\alpha(\omega)$ for the unique set in $\alpha$ to which $\omega$ belongs. Then, for any $\sigma$-algebra $\mathcal{G}$ we have

\begin{equation}\label{condout}
\frac{\mathbb{E}[f\mathbbm{1}_{\alpha(\omega)}\mid\mathcal{G}](\omega)}{\mathbb{E}[\mathbbm{1}_{\alpha(\omega)}\mid\mathcal{G}](\omega)}=\mathbbm{E}[f\mid\alpha\vee\mathcal{G}]\footnotemark
\end{equation}
\footnotetext{The proof of this simple fact is as follows. It suffices to consider a measurable set $A$, a $\mathcal{G}$-measurable set $G$ and an $\alpha$-measurable set $B$. Then
\begin{align}\int\mathbbm{1}_A\mathbbm{1}_B\mathbbm{1}_Gd\mu&=\int\frac{\mu(B\mid\mathcal{G})}{\mu(B\mid\mathcal{G})}\mathbb{E}[\mathbbm{1}_A\mathbbm{1}_B\mathbbm{1}_G\mid\mathcal{G}]d\mu=\int \mu(B\mid\mathcal{G})\mathbb{E}\left[\frac{\mathbbm{1}_A\mathbbm{1}_G\mathbbm{1}_B} {\mu(B\mid\mathcal{G})}\bigg\rvert\mathcal{G}\right]d\mu\nonumber\\
&=\int\mathbbm{1}_B\frac{\mathbb{E}[\mathbbm{1}_A\mathbbm{1}_G\mathbbm{1}_B\mid\mathcal{G}]}{\mu(B\mid\mathcal{G})}=\int \mathbbm{1}_B \mathbbm{1}_G\frac{\mu(A\cap \alpha(\cdot) \mid\mathcal{G})}{\mu(\alpha(\cdot)\mid\mathcal{G})}d\mu\nonumber
\end{align}}
Using this (for $\mathcal{G}=\mathcal{F}_{-\infty}^{0}\vee\mathcal{G}_{-\infty}^{0}$ and $\alpha=\mathcal{F}_1^{s_k(t)}$) and (\ref{shiftedcond}), we see that

$$\tilde{h}_{t,k}^i(\omega)=T^{s_k(t)}(\tilde{g}_{t,k}^i)=\mathbb{E}[T^{s_k(t)}(\mathbbm{1}_{\overline{[a_i]}})\mid \mathcal{F}_{-\infty}^{s_k(t)}\vee\mathcal{G}_{-\infty}^{0}](\omega)=\frac{z_{t,k}^i(\omega)}{z_{t,k}(\omega)}$$
and, by Lemma \ref{lemagibs}, when $\omega,\omega^*\in \hat{Y}$ are such that $\omega^*\in B_{t,k}(\omega)$, where
$$B_{t,k}(\omega)=\{\omega'\in\hat{Y}: \forall i,0\leq i\leq s_k(t),\omega_i^1=\omega'^{1}_i;\forall i,-l_k(t)+1\leq i\leq 0,\omega_i=\omega'_i\}$$
we have that
$$|z_{t,k}^i(\omega)-z_{t,k}^i(\omega^*)|\leq\gamma_{l_k(t)-1}z_{t,k}^i(\omega^*)$$
and
$$|z_{t,k}(\omega)-z_{t,k}(\omega^*)|\leq\gamma_{l_k(t)-1}z_{t,k}(\omega^*)$$
and therefore
\begin{align}|\tilde{h}_{t,k}^i(\omega)-\tilde{h}_{t,k}^i(\omega^*)|&=\left|\frac{z_{t,k}^i(\omega)}{z_{t,k}(\omega)}-\frac{z_{t,k}^i(\omega^*)}{z_{t,k}(\omega^*)}\right|\nonumber\\
&= \left|\frac{z_{t,k}^i(\omega)(z_{t,k}(\omega^*)-z_{t,k}(\omega))-z_{t,k}(\omega)(z_{t,k}^i(\omega^*)-z_{t,k}^i(\omega)}{z_{t,k}(\omega)z_{t,k}(\omega^*)}\right|\nonumber\\
&\leq 2\gamma_{l_k(t)-1}
\end{align}

Recall that $h^i_{t,k}(\omega)=\mathbb{E}[T^{s_k(t)}(\mathbbm{1}_{[a_i]})\mid \mathcal{F}^{s_k(t)}_{-l_k(t)+1}\vee\mathcal{G}_{-l_k(t)+1}^{0}]$. Then, using Rokhlin's disintegration (see \cite[Theorem 5.8]{Furst81}) against the $\sigma$-algebra $ \mathcal{H}_{t,k}:=\mathcal{F}^{s_k(t)}_{-\infty}\vee\mathcal{G}_{-\infty}^{0}$ we get
\begin{align}
h^i_{t,k}(\omega)&=\frac{\mu(T^{-s_k(t)}(\overline{[a_i]})\cap B_{t,k}(\omega))}{\mu(B_{t,k}(\omega))}=\nonumber\\
&=\int_{} \frac{\mu(T^{-s_k(t)}(\overline{[a_i]})\cap B_{t,k}(\omega)\mid\mathcal{H}_{t,k})(\omega^*)}{\mu(B_{t,k}(\omega))}d\mu(\omega^*)
\end{align}

and then, since $B_{t,k}(\omega)$ is $\mathcal{H}_{t,k}$-measurable:

\begin{align}|\tilde{h}_{t,k}^i(\omega)-h_{t,k}^i(\omega)|&=\left|\tilde{h}_{t,k}^i(\omega)-\int_{B_{t,k}(\omega)} \frac{\mu(T^{-s_k(t)}(A)\mid\mathcal{H}_{t,k})(\omega^*)}{\mu(B_{t,k}(\omega))}d\mu(\omega^*)\right| =\nonumber\\
&=\left|\int_{B_{t,k}(\omega)}\frac{\tilde{h}_{t,k}^i(\omega)-\tilde{h}^i_{t,k}(\omega^*)}{\mu(B_{t,k}(\omega))}d\mu(\omega^*)\right|\leq 2\gamma_{l_k(t)-1}\nonumber
\end{align}
%Now, noticing that $\mu(fB_{t,k}(\omega)\mid\mathcal{H}_{k,t})=$
%\begin{align}&\int \frac{\mu(T^{-s_k(t)}(A)\cap B_{t,k}(\omega)\mid\mathcal{H}_{t,k})(\omega^*)}{\mu(B_{t,k}(\omega))}d\mu(\omega^*)=\nonumber\\
%&\int_{B_{t,k}(\omega)} \frac{\mu(T^{-s_k(t)}(A)\mid\mathcal{H}_{t,k})(\omega^*)}{\mu(B_{t,k}(\omega))}d\mu(\omega^*)=\nonumber\\
%&\int_{B_{k,t}(\omega)} \frac{\mu(T^{-s_k(t)}(A)\cap B_{t,k}(\omega)\mid\mathcal{H}_{t,k})(\omega^*)}{\mu(B_{t,k}(\omega)\mid\mathcal{H}_{t,k})}\frac{\mu(B_{t,k}(\omega)\mid\mathcal{H}_{t,k})}{\mu(B_{t,k}(\omega))}d\mu(\omega^*)
%\end{align}
Since $\gamma_{l_k(t)}$ goes to zero, we have $|\tilde{h}_{t,k}^i(\omega)-h_{t,k}^i(\omega)|\to 0$ a.s.

We now claim that convergence of (\ref{eqav}) will follow from convergence of 

\begin{align}\label{ecuacioncita}&\frac{1}{N}\sum_{k=1}^N\left[\mathbbm{1}_I(R_\alpha^k(t))T^{k}(\mathbbm{1}_{\overline{[c]}})(\omega')\ T^{l_k(t)}\left(\tilde{f}\prod_{i=1}^r\tilde{h}_{t,k}^i\right)(\omega')\right]=\nonumber\\
&\frac{1}{N}\sum_{k=1}^N\left[\mathbbm{1}_I(R_\alpha^k(t))T^{l_k(t)}(\tilde{f})(\omega')\ T^{k}\left(\mathbbm{1}_{\overline{[c]}}\prod_{i=1}^r\tilde{g}_{t,k}^i\right)(\omega')\right]
\end{align}

Indeed, to see this claim is true we just use multilinearity as in the proof of Observation \ref{obse2}

\begin{align}\label{eculoca3}
&\limsup_N\left|\frac{1}{N}\sum_{k=1}^N\left[\mathbbm{1}_I(R_\alpha^k(t))T^k(\mathbbm{1}_{\overline{[c]}})T^{l_k(t)}\left(\mathbbm{1}_{\overline{[d]}}\prod_{i=1}^r(f_{i}^*)h_{t,k}^i-\tilde{f}\prod_{i=1}^r\tilde{h}_{t,k}^i\right)\right]\right|=\nonumber\\
&=\limsup_N\left|\sum_{j=1}^r\frac{1}{N}\sum_{k=1}^N\mathbbm{1}_I(R_\alpha^k(t))T^k(\mathbbm{1}_{\overline{[c]}})T^{l_k(t)}\psi^{j}_{t,k}\right|
\end{align}
where
$$\psi^{j}_{t,k}=\mathbbm{1}_{\overline{[d]}}\left(\prod_{i=1}^{j-1}f_{i}^*h_{t,k}^i\right)f_j^*(h_{t,k}^j-\tilde{h}_{t,k}^j)\left(\prod_{i=j+1}^rf_{i}^*\tilde{h}_{t,k}^i\right)$$

Then the fact that $|h_{t,k}^j-\tilde{h}_{t,k}^j|\to 0$ a.s. as $k\to\infty$ implies that $\psi^{k,j}_t\to 0$ a.s. as $k\to\infty$ and convergence of (\ref{eculoca3}) follows from Corollary \ref{coro2}.

Now, to prove convergence of (\ref{ecuacioncita}) notice that by Doob's theorem $\tilde{g}_{t,k}^i\to g_i=\mathbb{E}[\mathbbm{1}_{\overline{[a_i]}}\mid\mathcal{F}_\mathcal{G}]$ for a.e. $\omega'$, where $\mathcal{F}_{\mathcal{G}}=\bigcap_{k=1}^\infty \mathcal{F}_{-\infty}^0\vee\mathcal{G}_{-\infty}^{-k}$. We then obtain, using once more multilinearity and Corollary \ref{coro2} as in the proof of Observation \ref{obse2}, that the averages
\begin{align}
&\frac{1}{N}\sum_{k=1}^N\left[\mathbbm{1}_I(R_\alpha^k(t))T^{l_k(t)}(\tilde{f})(\omega')\left[T^k\left(\mathbbm{1}_{\overline{[c]}}\prod_{i=1}^r\tilde{g}^i_{t,k}\right)-T^k\left(\mathbbm{1}_{\overline{[c]}}\prod_{i=1}^rg_i\right)\right](\omega')\right]\nonumber
\end{align}
converge a.s. to 0.

Finally, from Corollary \ref{coroteo2} and the fact (Lemma \ref{lemagibbsexact}) that Gibbs invariant measures for H\"{o}lder potentials on topologically transitive subshifts of finite type are exact (and their natural extensions are K systems), we have, for each $t\in[0,1)$, a full $\mu$-measure set $\Omega'\subseteq \hat{Y}$ such that for all $\omega'\in\Omega'$,

$$\frac{1}{N}\sum_{k=1}^N\left[\mathbbm{1}_I(R_\alpha^k(t))T^{l_k(t)}(\tilde{f})(\omega')\ T^k\left(\mathbbm{1}_{\overline{[c]}}\prod_{i=1}^rg_i\right)(\omega')\right]\to|I|\mathbb{E}[\tilde{g}]\mathbb{E}[\tilde{f}]$$
\end{proof}
\subsection{Ergodicity}\label{secergodicity}

Let $Q$ be the CP-chain distribution generated by $\mu$ according to Theorem \ref{teoconvergence}. In this section we will show that $Q$ is ergodic for the transformation $S$.

Let $\tilde{Q}$ be the $\Xi$ component of $Q$ and $\tilde{\mathcal{S}}$ be the functions in $\mathcal{S}$ not depending on $t$ (equivalently, $I=\mathbb{T}$ in the definition of $\mathcal{S}$). For $(\omega,\nu)\in\Xi$ write $$\tilde{S}_{t,k}(\omega,\nu)=(T_{t,k}(\omega),S_{t,k}(\omega,\nu))$$
\begin{observation}\label{obsemix}$Q=\lambda\times\tilde{Q}$, where $\lambda$ is Lebesgue measure.
\end{observation}
\begin{proof}Any $f\in\mathcal{S}$ can be written $f=\mathbbm{1}_If'$ where $f'\in\tilde{\mathcal{S}}$ and $I\in\mathcal{I}$ and by the explicit limit in Theorem \ref{teoconvergence} it follows that $Q(f)=\lambda(I)\tilde{Q}(f')$
\end{proof}

We will also write, for measurable $f:\Xi\to\RR$
$$\tilde{S}_{t,h}(f)=f\circ\tilde{S}_{t,h}$$
and we will also abuse notation and write $S^k(f)$ with $f$ regarded as a function on $\mathbb{T}\times\Xi$ rather than $\Xi$.

\begin{lemma}\label{lema12}Let $a\in\Lambda^*$, $b\in\Lambda'^*$ and $f(\omega,\nu)=\nu([a]\times[b])=\phi_{[a]\times[b]}(\nu)$. Then
$$S^k(\tilde{S}_{s,h}(f))(t,\omega,\mu)=T^k\left(C_{(s,h),(t,k)}^{[a]\times[b]}+D_{(s,h),(t,k)}^{[a]\times[b]}\right)(\omega')$$
where $\omega'$ is any element of $\pi^{-1}(\{\omega\})$, $C_{(s,h),(t,k)}^{[a]\times[b]}$ is the function given by
$$\mathbb{E}\left[\mathbbm{1}_{T^{s_k(t)-l_h(s)}(\overline{[b]})}\mid \mathcal{F}_{-k+1}^{h}  \vee\mathcal{G}_{-k+1}^{-s_k(t)+l_h(s)}\right] \mathbb{E}\left[\mathbbm{1}_{T^{-h}(\overline{[a]})}\mid \mathcal{F}_{-k+1}^{h}  \vee\mathcal{G}_{-k+1}^{-s_k(t)+l_h(s)}\right]$$
and $$D_{(s,h),(t,k)}^{[a]\times[b]}\to 0\quad\text{a.s. in }\Omega$$
\end{lemma}
\begin{proof}We have that
$$S^k(\tilde{S}_{s,h}(f))(t,\omega,\mu)=f\left(S^k(t,\tilde{S}_{s,h}(\omega,\mu))\right)$$
and
$$S^k(t,\tilde{S}_{s,h}(\omega,\mu))=\left(R_\alpha^k(t),(T^{k+h}(\omega^1),T^{l_k(t)+l_h(s)}(\omega^2)),S_{t,k}\left(T_{s,h}(\omega),\mu_{s,h}^\omega\right)\right)$$
And from Observation \ref{obsemedidas} and invariance of $\mu$ (that is, (\ref{shiftedcond}))
\begin{align}\label{shiftedterm}
S_{t,k}&\left(T_{s,h}(\omega),\mu_{s,h}^\omega\right)([a]\times[b])=\mathbb{E}\left[T^{k+h}(\mathbbm{1}_{\overline{[a]}})T^{l_{k,h}(t,s)}(\mathbbm{1}_{\overline{[b]}})\mid \mathcal{F}_1^{k+h}\vee\mathcal{G}_1^{l_{k,h}(t,s)}\right](\omega)\nonumber\\
&=T^k\left( \mathbb{E}\left[\mathbbm{1}_{T^{-h}(\overline{[a]})}\mathbbm{1}_{T^{k-l_{k,h}(t,s)}(\overline{[b]})}\mid \mathcal{F}_{-k+1}^{h}  \vee\mathcal{G}_{-k+1}^{-k+l_{k,h}(t,s)}\right]\right)(\omega')
\end{align}
where $l_{k,h}(t,s)=l_k(t)+l_h(s)$.

Now, recall that in Lemma \ref{lema1} (and the preceding (\ref{ecusplit})) we managed to write the conditional expectation 
$$F^{[a]\times [b]}_{t,k}=\mathbb{E}\left[\mathbbm{1}_{T^{-s_k(t)}(\overline{[a]})}\mathbbm{1}_{\overline{[b]}}\mid\mathcal{F}^{s_k(t)}_{-l_k(t)+1}\vee\mathcal{G}^0_{-l_k(t)+1}\right]$$
as a sum of two terms $C^{[a]\times [b]}_{t,k}$ and $D^{[a]\times [b]}_{t,k}$, where $D^{[a]\times [b]}_{t,k}$ converged to 0 a.s. and

$$C^{[a]\times [b]}_{t,k}=\mathbb{E}\left[\mathbbm{1}_{T^{-s_k(t)}(\overline{[a]})}\mid\mathcal{F}_{-l_k(t)+1}^{s_k(t)}\vee\mathcal{G}^0_{-l_k(t)+1}\right]\mathbb{E}\left[\mathbbm{1}_{[b]}\mid\mathcal{F}_{-l_k(t)+1}^{s_k(t)}\vee\mathcal{G}^0_{-l_k(t)+1}\right]$$

Using projections in exactly the same way as in (\ref{ecusplit}) we can write the term that is being shifted in (\ref{shiftedterm}) as
$$\mathbb{E}\left[\mathbbm{1}_{T^{-h}(\overline{[a]})}\mathbbm{1}_{T^{s_k(t)-l_h(s)}(\overline{[b]})}\mid \mathcal{F}_{-k+1}^{h}  \vee\mathcal{G}_{-k+1}^{-s_k(t)+l_h(s)}\right]=C_{(s,h),(t,k)}^{[a]\times[b]}+D_{(s,h),(t,k)}^{[a]\times[b]}$$
and then prove that $D_{(s,h),(t,k)}^{[a]\times[b]}\to 0$ a.s. in $\Omega$.
\end{proof}

\begin{theorem}\label{teomixing}Let $(\Omega,\mu,T)$ be as in Theorem \ref{teoconvergence}, $Q$ be the CP-chain distribution generated by $\mu$ and $f,f^*\in\tilde{\mathcal{S}}$, then
$$\mathbb{E}_Q\left[f\tilde{S}_{s,h}(f^*)\right]\xrightarrow{h\to\infty}\mathbb{E}_Q[f]\mathbb{E}_Q[f^*]$$
\end{theorem}
\begin{proof}Since $f,f^*\in\tilde{\mathcal{S}}$, we may write
$$f(\omega,\nu)=\mathbbm{1}_{[c]}(\omega^1)\mathbbm{1}_{[d]}(\omega^2)\prod_{i=1}^r\phi_{[a_i]\times[b_i]}(\nu)$$
and
$$f^*(\omega,\nu)=\mathbbm{1}_{[c^*]}(\omega^1)\mathbbm{1}_{[d^*]}(\omega^2)\prod_{i=1}^{r'}\phi_{[a_i^*]\times[b_i^*]}(\nu)$$
$\mathcal{S}\subseteq C(\Xi)$ so by Theorem \ref{teoconvergence} we have
$$\mathbb{E}_Q\left[f\tilde{S}_{s,h}(f^*)\right]=\lim_N\frac{1}{N}\sum_{k=1}^NS^k\left(f\tilde{S}_{s,h}(f^*)\right)(0,\omega,\mu)$$
for a.e. $\omega$, and by Lemma \ref{lema12} (writing, as before, $l_k=l_k(0)$)
\begin{align}
S^k&\left(f\tilde{S}_{s,h}(f^*)\right)(0,\omega,\mu)=T^k\left(\mathbbm{1}_{\overline{[c]}}\mathbbm{1}_{T^{-h}(\overline{[c^*]})}\right)(\omega)T^{l_{k}}\left(\mathbbm{1}_{\overline{[d]}}\mathbbm{1}_{T^{-l_{h}(s)}(\overline{[d^*]})}\right)(\omega)\nonumber\\
&\prod_{i=1}^r\mathbb{E}\left[T^k(\mathbbm{1}_{\overline{[a_i]}})T^{l_k}(\mathbbm{1}_{\overline{[b_i]}})\mid\ \mathcal{F}_1^k\vee\mathcal{G}_1^{l_k}\right](\omega')\prod_{j=1}^{r'}S^k(\tilde{S}_{s,h}(\phi_{[a_j^*]\times[b_j^*]}))(0,\omega,\mu)=\nonumber\\
&=T^k\left(\mathbbm{1}_{\overline{[c]}}\mathbbm{1}_{T^{-h}(\overline{[c^*]})}\right)(\omega)T^{l_{k}}\left(\mathbbm{1}_{\overline{[d]}}\mathbbm{1}_{T^{-l_{h}(s)}(\overline{[d^*]})}\right)(\omega)\nonumber\\
&T^k\left(\prod_{i=1}^r\left(C_{0,k}^{[a_i]\times[b_i]}+D_{0,k}^{[a_i]\times[b_i]}\right)\prod_{j=1}^{r'}\left(C_{(s,h),(0,k)}^{[a_j^*]\times[b_j^*]}+D_{(s,h),(0,k)}^{[a_j^*]\times[b_j^*]}\right)\right)(\omega')
\end{align}
for any $\omega'\in\pi^{-1}(\{\omega\})$.

From here we follow the proofs of Corollary \ref{coro1}, Observation \ref{obse2} and Theorem \ref{teoconvergence} to show that

$$\frac{1}{N}\sum_{k=1}^NS^k \left(f\tilde{S}_{s,h}(f^*)\right)(0,\omega,\mu)\xrightarrow{a.s.} \mathbb{E}[\tilde{f}_h]\mathbb{E}[\tilde{g}_h]$$
where, writing as before $\mathcal{F}_\mathcal{G}=\bigcap_{k=1}^\infty\mathcal{F}_{-\infty}^0\vee\mathcal{G}_{-\infty}^{-k}$,
\begin{align}
\tilde{g}_h&=\mathbbm{1}_{\overline{[c]}\cap T^{-h}(\overline{[c^*]})}\prod_{i=1}^r\mathbb{E}\left[\mathbbm{1}_{\overline{[a_i]}}\biggr\rvert\mathcal{F}_\mathcal{G}\right]\prod_{j=1}^{r'}\mathbb{E}\left[\mathbbm{1}_{T^{-h}(\overline{[a_j^*]})}\biggr\rvert\bigcap_{k=1}^\infty\mathcal{F}_{-\infty}^h\vee\mathcal{G}_{-\infty}^{-k}\right]\nonumber\\
&=\mathbbm{1}_{\overline{[c]}}\prod_{i=1}^r\mathbb{E}\left[\mathbbm{1}_{\overline{[a_i]}}\biggr\rvert\mathcal{F}_{\mathcal{G}}\right]T^h\left(\mathbbm{1}_{\overline{[c^*]}}\prod_{j=1}^{r'}\mathbb{E}\left[\mathbbm{1}_{\overline{[a_j^*]}}\biggr\rvert\mathcal{F}_{\mathcal{G}}\right]\right)\nonumber
\end{align}
and, writing $\mathcal{F}_\infty=\mathcal{F}_{-\infty}^\infty\vee\mathcal{G}_{-\infty}^{0}$ (compare to Theorem \ref{teoconvergence} and recall that $\Pi_\infty(f)=\mathbb{E}[f\mid\mathcal{F}_\infty]$)
\begin{align}
\tilde{f}_h&=\mathbbm{1}_{\overline{[d]}\cap T^{-l_h(s)}(\overline{[d^*]})}\prod_{i=1}^r\mathbb{E}\left[\mathbbm{1}_{\overline{[b_i]}}\mid \mathcal{F}_\infty\right]\prod_{j=1}^{r'}\mathbb{E}\left[\mathbbm{1}_{T^{l_h(s)}(\overline{[b_j^*]})}\mid \mathcal{F}_{-\infty}^\infty\vee\mathcal{G}_{-\infty}^{l_h(s)}\right]\nonumber\\
&=\mathbbm{1}_{\overline{[d]}}\prod_{i=1}^r\mathbb{E}\left[\mathbbm{1}_{\overline{[b_i]}}\mid \mathcal{F}_\infty\right]T^{l_h(s)}\left(\mathbbm{1}_{\overline{[d^*]}}\prod_{j=1}^{r'}\mathbb{E}\left[\mathbbm{1}_{\overline{[b_j^*]}}\mid \mathcal{F}_\infty\right]\right)\nonumber
\end{align}
Hence, $\mathbb{E}_Q\left[f\tilde{S}_{s,h}(f^*)\right]=\mathbb{E}[\tilde{f}_h]\mathbb{E}[\tilde{g}_h]=\mathbb{E}[f_1T^{l_h(s)}(f_2)]\mathbb{E}[g_1T^h(g_2)]$, where 
\begin{align}
&f_1=\mathbbm{1}_{\overline{[d]}}\prod_{i=1}^r\mathbb{E}\left[\mathbbm{1}_{\overline{[b_i]}}\mid \mathcal{F}_\infty\right]\quad\quad\quad&f_2=\mathbbm{1}_{\overline{[d^*]}}\prod_{j=1}^{r'}\mathbb{E}\left[\mathbbm{1}_{\overline{[b_j^*]}}\mid \mathcal{F}_\infty\right]\nonumber\\
&g_1=\mathbbm{1}_{\overline{[c]}}\prod_{i=1}^r\mathbb{E}\left[\mathbbm{1}_{\overline{[a_i]}}\biggr\rvert\mathcal{F}_\mathcal{G}\right]
&g_2=\mathbbm{1}_{\overline{[c^*]}}\prod_{j=1}^{r'}\mathbb{E}\left[\mathbbm{1}_{\overline{[a_j^*]}}\biggr\rvert\mathcal{F}_\mathcal{G}\right]\nonumber
\end{align}

Finally, since Gibbs systems are mixing,
\begin{align}\mathbb{E}_Q\left[f\tilde{S}_{s,h}(f^*)\right]=\mathbb{E}[f_1T^{l_h(s)}(f_2)]\mathbb{E}[g_1T^h(g_2)]\xrightarrow{h\to\infty}\nonumber \mathbb{E}\left[f_1 \right]\mathbb{E}\left[f_2 \right]\mathbb{E}\left[g_1  \right]\mathbb{E}\left[g_2\right]\nonumber
\end{align}
and again from Theorem \ref{teoconvergence}, for every $\omega$ in a full measure $\Omega'_+\subseteq\Omega_+$
$$\mathbb{E}_Q[f]=\lim_N\frac{1}{N}\sum_{k=1}^NS^k(f)(0,\omega,\mu)=\mathbb{E}[f_1]\mathbb{E}[g_1]$$
and
$$\mathbb{E}_Q[f^*]=\lim_N\frac{1}{N}\sum_{k=1}^NS^k(f^*)(0,\omega,\mu)=\mathbb{E}[f_2]\mathbb{E}[g_2]$$

So $\lim_h\mathbb{E}_Q[f\tilde{S}_{t,h}(f^*)]= \mathbb{E}_Q[f^*]\mathbb{E}_Q[f]$, as wanted.
\end{proof}
\begin{corollary}$Q$ is an invariant and ergodic measure for the transformation $S$ on $X$.
\end{corollary}
\begin{proof}Invariance follows from the definition of $Q$ as the limit functional of the ergodic averages in  (\ref{genericity}). To show ergodicity take $f_1,f_2\in\mathcal{S}$ and write $f_i=\mathbbm{1}_{I_i}f'_i$ for $i=1,2$ and $f_i'\in\tilde{\mathcal{S}}$, $I_i\in\mathcal{I}$. Then, by Observation \ref{obsemix},
\begin{align}
\mathbb{E}_Q[f_1S^{-k}(f_2)]&=\int_X\mathbbm{1}_{I_1}(t)\mathbbm{1}_{I_2}(R_\alpha^k(t))f_1'(\omega,\nu)\tilde{S}_{t,k}(f_2')(\omega,\nu)\ d(\lambda\times\tilde{Q})(t,\omega,\nu)\nonumber\\
&=\int\mathbbm{1}_{I_1}(t)\mathbbm{1}_{I_2}(R_\alpha^k(t))\mathbb{E}_{\tilde{Q}}[f_1'\tilde{S}_{t,k}(f_2')]\ d\lambda(t)\nonumber
\end{align}
And then, by the dominated convergence theorem (recall $f_1$ and $f_2$ are bounded) and Theorem \ref{teomixing}
$$\mathbb{E}_Q[f_1S^{-k}(f_2)]-\lambda(I_1\cap R^{-k}_\alpha(I_2))\mathbb{E}_{\tilde{Q}}[f_1']\mathbb{E}_{\tilde{Q}}[f_2']\to 0$$
So, by ergodicity of $R_\alpha$
\begin{align}
\lim_N\frac{1}{N}\sum_{k=1}^N\mathbb{E}_Q[f_1S^{-k}(f_2)]=&\lim_N\frac{1}{N}\sum_{k=1}^N\lambda(I_1\cap R^{-k}_\alpha(I_2))\mathbb{E}_{\tilde{Q}}[f_1']\mathbb{E}_{\tilde{Q}}[f_2']=\nonumber\\
&=\mathbb{E}_{\tilde{Q}}[f_1']\mathbb{E}_{\tilde{Q}}[f_2']\lambda(I_1)\lambda(I_1)=\mathbb{E}_Q[f_1]\mathbb{E}_Q[f_2]\nonumber
\end{align}
By bilinearity this can be extended to any $f_1,f_2\in\mathcal{A}$, and since any function in $C(X)$ can be uniformly approximated by functions in $\mathcal{A}$ it follows that
$$\mathbb{E}_Q[f_1S^{-k}(f_2)]\to \mathbb{E}_Q[f_1]\mathbb{E}_Q[f_2]$$
for all $f_1,f_2$ in $C(X)$ and by density of $C(X)$ for all functions in $L^1(X)$.
\end{proof}
\subsection{Proof of the projection result}
To prove Theorem \ref{projectionresult} it suffices to invoke the following result, which is proven in \cite[Lemma 5.1]{Ferguson2013}. Notice that even though it is stated there for Bernoulli measures, the result holds in general, since their proof only uses Marstrand's projection theorem for measures (which holds in general), the fact that the $Q$ distribution is a product of Lebesgue measure and its $\Xi$ component $\tilde{Q}$ and the fact that a distribution generated by $\mu$ is supported on measures $\nu$ that satisfy $\text{dim }\nu=\text{dim }\mu$ (this is proved in the second part of \cite[Lemma 7.9]{Hochman2010}).
\begin{lemma}For any $\mu\in\mathcal{P}(\Sigma^\NN)$ that generates a CP-chain distribution on $\mathbb{T}\times\Xi$ of the form $Q=\lambda\times\tilde{Q}$ and for every $\pi\in\Pi_{2,1}\setminus\{\pi_1,\pi_2\}$ we have
$$E(\pi)=\min\{1,\text{dim }\pi\}$$
\end{lemma}

Theorem \ref{projectionresult} then follows from Observation \ref{obsemix} and part c) of Theorem \ref{teohochman}.

\section{Concluding remarks}Theorem \ref{projectionresult} has a trivial corollary for sets, namely, that any subset of $[0,1]^2$ whose dimension may be approximated by the dimension of Gibbs measures supported on it satisfies dimension conservation for one-dimensional projections. The most natural examples of such subsets are the images of subshifts of finite type under the $(m,n)$-adic encoding. However, the dimension of such sets may be readily approximated by the dimension of Bernoulli measures, as shown in \cite{Peres96}, and hence dimension conservation for their projections was a trivial corollary of the main result in \cite{Ferguson2013}.

It must be mentioned, nonetheless, that for the distance set result in \cite{Ferguson2013} related to Falconer's conjecture our main theorem does imply a marginal improvement. Given some $K\subseteq[0,1]^2$ which arises as the pushforward of a topologically transitive subshift of finite type as above, if we wanted to use \cite[Theorem 1.7]{Ferguson2013} to ensure that $D(K)$ (the distance set of $K$) is 1 it was not enough to ask that $\text{dim}(K)\geq 1$ and $\mathcal{H}^1(K)>0$. Indeed, for $\text{dim}(K)= 1$ the approximating Bernoulli measures $\mu_\epsilon$ that generate an ergodic CP-chain would have $\text{dim}(\mu_\epsilon)<1$ so that clearly $\mathcal{H}^1(\text{supp }\mu_\epsilon)=0$, which falls short of the conditions for \cite[Theorem 1.7]{Ferguson2013}. With our results, however, we now know that the Gibbs measure $\mu$ that $K$ supports and which has full dimension (i.e. $\text{dim}(K)=\text{dim}(\mu)$) generates an ergodic CP-chain so that if $\mathcal{H}^1(K)>0$ we must have $D(K)=1$ by \cite[Theorem 1.7]{Ferguson2013}.

%While  \cite{Peres94}, but for $K$ the pushforward of some topologically transitive subshift of finite type with $\text{dim }K= 1$ there is no reason to believe that $\mathcal{H}^1(K)$ is positive.

It would be interesting to extend our result on Gibbs measures to more general systems, in particular to exact endomorphisms, for which the convergence of the relevant multiple ergodic average is known (see \cite{Derri08}) and for which all our reductions previous to the proof of Theorem \ref{teoconvergence} may be carried out. Yet, this is no easy task. To see why, notice that that if $\mu$ generates a CP-chain then, for any cylinder $a\in\Lambda^*$, the averages
$$\frac{1}{N}\sum_{k=0}^{N-1}T^k\left(\mathbb{E}_\mu[\mathbbm{1}_{[a]\times\Lambda'^\ZZ}\mid\mathcal{F}_{-k+1}^0\vee\mathcal{G}_{-k+1}^{-s_k}]\right)(\omega)$$
must convergence for $\mu$-a.e. $\omega$ . In the Bernoulli case this convergence is trivial, since independence from the past means the average is constantly equal to $\mu([a]\times\Lambda'^\ZZ)$. In the case of a product of two measures on $\Lambda^\ZZ$ and $\Lambda'^\ZZ$ (as in \cite{Hochman2010}), the conditioning of $\mathcal{G}$ disappears,  $\mathcal{F}_{-k+1}^0\vee\mathcal{G}_{-k+1}^{-s_k}$ becomes the monotone sequence $\mathcal{F}_{-k+1}^0$, which converges pointwise, and Maker's theorem ensures the convergence of the averages. In the general case, this monotonicity is lost, and the sequence of $\sigma$-algebras $\mathcal{H}_k=\mathcal{F}_{-k+1}^0\vee\mathcal{G}_{-k+1}^{-s_k}$ need not even satisfy the condition
$$\mathcal{H}^*=\bigcap_{k=1}^\infty\bigcup_{l\geq k}\mathcal{H}_l=\bigcup_{k=1}^\infty\bigcap_{l\geq k}\mathcal{H}_l=\mathcal{H}_*$$
which is sufficient (though not necessary) for the convergence of $\mathbb{E}[f\mid\mathcal{H}_k]$ in $L^2$-norm, but is not sufficient for its pointwise convergence (see \cite{Alonso88}). %For a simple example in which $\mathcal{H}^*\neq\mathcal{H}_*$\footnote{I thank Mike Hochman for this example.}, let $m=2$, $n=3$, let the $\Lambda'$-coordinate be a Bernoulli scheme with probabilities $\{1/4,1/2,1/4\}$ for $\{0,1,2\}$, and let us consider the joining with $\Lambda^\NN=\{0,1\}^\NN$ given by $\omega^1_1=\omega^2_1\mod 2$ and $\omega^1_{i+1}=\omega^2_i+\omega^2_{i+1}\mod 2$. This system is invariant and for the natural extension of this system we have that $\mathcal{H}^*=\mathcal{F}_{-\infty}^0\vee\mathcal{G}_{-\infty}^0$ but $\mathcal{H}_*=\mathcal{F}_{-\infty}^0$.

In the present work we managed to sort out this problem by showing the sequence  $\mathbb{E}[\mathbbm{1}_{\overline{[a]}}\mid \mathcal{H}_k]$ is asymptotically close to some other sequence $\mathbb{E}[\mathbbm{1}_{\overline{[a]}}\mid \mathcal{T}_k]$ for which the $\mathcal{T}_k$ are indeed monotonic. In the absence of the Gibbs property, this approach looks rather ineffective.
\bibliography{Gibbsbiblio}
\bibliographystyle{plain}
\end{document}